\documentclass{amsart}

\usepackage[T1]{fontenc} 

\usepackage{ adjustbox }
\usepackage[boxed]{ algorithm2e }
\usepackage{ amsthm }
\usepackage{ amssymb }
\usepackage{ amsmath }
\usepackage{ array }
\usepackage{ booktabs }
\usepackage{ cite }
\usepackage{ color }
\usepackage{ enumitem }
\usepackage{ graphicx }
\usepackage{ latexsym }
\usepackage{ mdframed }
\usepackage[most]{ tcolorbox }
\usepackage{ url }
\usepackage[all]{ xy }

\theoremstyle{definition}

\newtheorem{definition}{Definition}[section]
\newtheorem{conjecture}[definition]{Conjecture}
\newtheorem{example}[definition]{Example}

\newtheorem{openproblem}[definition]{Open problem}

\newtheorem{remark}[definition]{Remark}

\theoremstyle{plain}

\newtheorem{lemma}[definition]{Lemma}
\newtheorem{proposition}[definition]{Proposition}
\newtheorem{theorem}[definition]{Theorem}
\newtheorem{corollary}[definition]{Corollary}

\newcommand{\au}{\mathcal{A\,U}}
\newcommand{\na}{\mathcal{N\!A}}
\newcommand{\rb}{\mathcal{RB}}
\newcommand{\tad}{\mathcal{T\!AD}}
\newcommand{\taq}{\mathcal{T\!AQ}}

\newcommand{\longleftmapsto}{\text{\reflectbox{$\longmapsto$}}}

\makeatletter
\renewcommand{\tocsection}[3]{%
\indentlabel{\@ifnotempty{#2}{\makebox[1.50em][l]{\ignorespaces#1#2.}}}#3}
\renewcommand{\tocsubsection}[3]{%
\indentlabel{\@ifnotempty{#2}{\hspace*{1.50em}\makebox[2.25em][l]{\ignorespaces#1#2.}}}#3}
\renewcommand{\tocsubsubsection}[3]{%
\indentlabel{\@ifnotempty{#2}{\hspace*{3.75em}\makebox[3.00em][l]{\ignorespaces#1#2.}}}#3}
\makeatother

\makeatletter
\@namedef{subjclassname@2020}{%
\textup{2020} Mathematics Subject Classification}
\makeatother

\allowdisplaybreaks

\begin{document}

\title{Ternary analogs of dendriform algebras and quadri-algebras}

\author{Yu Hin (Gary) Au}

\address{Department of Mathematics and Statistics,
University of Saskatchewan,
Canada}

\email{au@math.usask.ca}

\author{Murray R. Bremner}

\address{Department of Mathematics and Statistics,
University of Saskatchewan,
Canada}

\email{bremner@math.usask.ca}



\begin{abstract}
We construct ternary analogs of dendriform algebras and quadri-algebras,
generalizing the well-known binary case.
These new structures are defined in terms of nonsymmetric operads.
We construct these new operads using ternary operations defined in terms of 
ternary Rota-Baxter operators.
The relations defining these operads lie in the quadratic subspace of the kernel of a rewriting morphism 
from a free operad to the operad generated by one associative ternary operation 
and one (or two commuting) unary operators satisfying the ternary Rota-Baxter relation(s).
We use computational algebra to compute a basis for the quadratic component of the kernel 
and thereby obtain the quadratic relations which define these new operads.
In the dendriform case, we obtain 3 operations satisfying 10 relations.  
In the quadri-algebra case, we obtain 9 operations satisfying 50 relations. 
The latter computations combine modular and rational arithmetic, 
but we are able to provide certificates of the results which can be verified 
in characteristic 0. 
\end{abstract}

\maketitle

{\footnotesize\tableofcontents}


\section{Introduction}

The results of this paper combine three topics, namely $n$-ary algebras, dendriform algebras,
and Rota-Baxter operators.
We use computational algebra to determine polynomial identities satisfied by 
ternary analogs of dendriform algebras and quadri-algebras.
In this section, we recall some background, motivation, and references.


\subsubsection*{Ternary and $n$-ary algebras}

An $n$-ary algebra ($n$-tuple system) is a vector space $A$ together with 
a multilinear map $A^n \to A$.
Hestenes \cite{Hestenes} introduced the concept of associative triple system;
this theory was further developed by Lister \cite{ListerTR}.
This work was extended to associative $n$-tuple systems by Carlsson \cite{Carlsson}.
Gnedbaye \cite{Gnedbaye} discusses the operads of associative $n$-ary algebras and their Koszul duals.

The notion of a Lie triple system was introduced by Jacobson \cite{Jacobson};
this theory was further developed by Lister \cite{ListerLTS}.
Two different definitions of Lie $n$-tuple systems were introduced by Filippov
\cite{Filippov} and Hanlon and Wachs \cite{HW}.
Kerner \cite{Kerner} discusses the application of ternary algebraic structures in physics.
For applications of $n$-ary analogs of Lie algebras in physics,
see de Azc\'arraga and Izquierdo \cite{DI}.

For Jordan triple systems, see Neher \cite{Neher}; for alternative triple systems,
see Loos \cite{Loos}; for Poisson triple systems, see Bremner and Elgendy \cite{BE},
and for Leibniz triple systems, see Bremner and S\'anchez-Ortega \cite{BSO}.
For a detailed survey of nonassociative triple systems, see the lecture notes of
Meyberg \cite{Meyberg}.

Duplij \cite{Duplij} provides a survey of polyadic ($n$-ary) algebraic structures.


\subsubsection*{Dendriform algebras and quadri-algebras}

Dendriform algebras were introduced by Loday \cite{Loday}; see also Loday and Ronco \cite{LR}.
Dendriform algebras may be obtained from associative algebras
by splitting the operation; see Bai et al.~\cite{BBGN}.
A dendriform algebra is an associative algebra $D$ whose 
multiplication $\bullet$ can be decomposed as the sum of two operations,
$x \bullet y = x \prec y + x \succ y$,
satisfying these relations:
\[
( x \succ y ) \prec z = x \succ ( y \prec z ),
\;
( x \prec y ) \prec z = x \prec ( y \bullet z ),
\;
x \succ ( y \succ z ) = ( x \bullet y ) \succ z.
\]
For cohomology and deformations of dendriform algebras, see Das \cite{Das}.
For the braided analog of dendriform algebras, see Li and Guo \cite{LG}.
Foissy \cite{F2} describes the connection between typed binary trees and generalized dendriform algebras.
For polydendriform algebras, see Giraudo \cite{Giraudo}.
Bai et al.~\cite{BLN} introduce L-dendriform algebras, a Lie-theoretic analog of dendriform algebras.
Hou and Bai \cite{HB} discuss the analog of dendriform algebras in Jordan theory.
Manchon \cite{Manchon} explains the connection between dendriform algebras and
Connes-Kreimer renormalization.

Splitting the operations in a dendriform algebra leads to the definition of quadri-algebra,
introduced by Aguiar and Loday \cite{AL}.
A quadri-algebra is an associative algebra $Q$ whose 
multiplication $\bullet$ can be decomposed as the sum of four operations,
$x \bullet  y = x \nearrow y + x \searrow y + x \swarrow y + x \nwarrow y$. 
Before stating the defining relations, we introduce four auxiliary operations,
\begin{alignat*}{2}
x \succ y &:= x \nearrow y + x \searrow y,
&\qquad
x \prec y &:= x \nwarrow y + x \swarrow y,
\\
x \wedge y &:= x \nwarrow y + x \nearrow y,
&\qquad
x \vee y &:= x \swarrow y + x \searrow y.
\end{alignat*}
The quadri-algebra relations are then as follows,
\small
\begin{alignat*}{3}
( x \nwarrow y ) \nwarrow z &= x \nwarrow ( y \bullet z ),
&\;
( x \nearrow y ) \nwarrow z &= x \nearrow ( y \prec z ),
&\;
( x \wedge y ) \nearrow z &= x \nearrow ( y \succ z ),
\\
( x \swarrow y ) \nwarrow z &= x \swarrow ( y \wedge z ),
&\;
( x \searrow y ) \nwarrow z &= x \searrow ( y \nwarrow z ),
&\;
( x \vee y ) \nearrow z &= x \searrow ( y \nearrow z ),
\\
( x \prec y ) \swarrow z &= x \swarrow ( y \vee z ),
&\;
( x \succ y ) \swarrow z &= x \searrow ( y \swarrow z ),
&\;
( x \bullet y ) \searrow z &= x \searrow ( y \searrow z ).
\end{alignat*}
\normalsize
The operations $\wedge$ and $\vee$ endow $Q$ with the structure of a dendriform algebra,
as do the operations $\prec$ and $\succ$.

For a study of the quadri-algebra operad, see Vallette \cite[\S\S 5.6-5.7]{V}.
Gr\"obner bases for the quadri-algebra operad are obtained by Madariaga \cite{M}.
For free quadri-algebras and dual quadri-algebras, see Foissy \cite{F}.
For quadri-bialgebras, see Ni and Bai \cite{NB1}.
For combinatorial aspects of quadri-algebras, see Foissy et al.~\cite{FMNT}.
For the connection between Hopf algebras and quadri-algebras, 
see Ebrahimi-Fard and Guo \cite{EFG2}.
For $\mathcal{O}$-operators on quadri-algebras, see Bai \cite{Bai}.

Further splitting of operations produces the octo-algebras introduced by Leroux \cite[\S 5]{Leroux};
see also Ni and Bai \cite{NB2}.
For an operadic point of view on splitting of operations, see Ebrahimi-Fard and Guo \cite{EFG}
and Pei et al.~\cite{PBG}.


\subsubsection*{Rota-Baxter operators}

Let $A$ be an associative algebra with product denoted $xy$.
A Rota-Baxter operator (of weight zero) on $A$ is a linear map 
$R\colon A \to A$ satisfying the Rota-Baxter relation,
\begin{equation}
\label{binaryRB}
R(x) R(y) - R( x R(y) ) - R( R(x) y ) = 0,
\qquad
\forall x, y \in A.
\end{equation}
The pair $(A,R)$ is called a Rota-Baxter algebra.
(We will frequently use the abbreviation RB for Rota-Baxter.)
The theory of RB operators originated in the work of Baxter \cite{Baxter}
and was developed from an algebraic point of view by Rota \cite{Rota}.
The monograph of Guo \cite{Guo} gives a systematic introduction to RB theory with many references.
RB algebras have appeared in theoretical physics in the work of
Connes and Kreimer on renormalization of perturbative quantum field theory;
see Ebrahimi-Fard and Guo \cite{EFG5}.

An important fact about RB operators 
was first noted by Aguiar \cite[Proposition 5.1]{Aguiar}:
if $A$ is an RB algebra, then the operations 
$x \prec y := x R(y)$ and $x \succ y := R(x) y$
endow $A$ with the structure of a dendriform algebra.
For more about the relation between RB algebras 
and dendriform algebras, see Ebrahimi-Fard and Guo \cite{EFG3,EFG4}.

Aguiar and Loday \cite[Corollary 2.6]{AL} showed that
if $R$ and $S$ are two commuting RB operators on an associative algebra $A$,
then the operations
\begin{alignat*}{2}
x \nearrow y &:= R(x) S(y),
&\qquad
x \searrow y &:= R(S(x)) y,
\\
x \swarrow y &:= S(x) R(y),
&\qquad
x \nwarrow y &:= x R(S(y)),
\end{alignat*}
endow $A$ with the structure of a quadri-algebra.


\subsubsection*{Results of this paper}

In \S\ref{preliminaries} we recall basic definitions and results about
nonsymmetric operads, Rota-Baxter operads, and the rewriting morphism.
In \S\ref{dimensionformula} we prove a dimension formula for the homogeneous 
subspaces of the operad $\au_{n,d}$ generated by one associative $n$-ary operation 
and $d$ commuting unary operators.
In \S\ref{twopropositions} we show that the binary RB operad
contains suboperads which are quotients of $n$-ary RB operads, 
and that in an $n$-ary RB operad the sum of the $n$ new $n$-ary operations is associative.
In \S\ref{ternarydendriform} we discuss the ternary analog of dendriform algebras
and determine the quadratic defining relations:
there are three ternary operations satisfying ten relations.
In \S\ref{ternaryquadrialgebras} we discuss the ternary analog of 
quadri-algebras:
there are nine ternary operations satisfying fifty relations.


\subsubsection*{Notation}

We write $\mathbb{F}$ for an arbitrary field.
All computations described in this paper take place over either the rational numbers 
$\mathbb{Q}$ (primarily) or a finite prime field $\mathbb{F}_p$ (for certain large computations).
We write $[n] = \{ 1, \dots, n \}$.
The computations described below were done using Maple 2020 on a 2026 MacBook Pro 
with the Apple M5 Pro chip using macOS Tahoe 26.6.


\section{Preliminaries}
\label{preliminaries}


\subsubsection*{Nonsymmetric operads}

We recall the partial definition of nonsymmetric operad;
that is, the definition in terms of partial compositions.

\begin{definition}
\cite[\S3.2.2]{BD} \,
A \emph{nonsymmetric operad} is a collection of vector spaces 
$\mathcal{O} = \{ \mathcal{O}(n) \mid n \ge 0 \}$
indexed by the non-negative integers 
together with an element $\mathrm{id} \in \mathcal{O}(1)$
and linear maps called \emph{partial compositions},
\[
\circ_i \colon \mathcal{O}(m) \otimes \mathcal{O}(n) \longrightarrow \mathcal{O}(m+n-1),
\qquad
\alpha \otimes \beta \longmapsto
\alpha \circ_i \beta,
\qquad
1 \le i \le m,
\]
which satisfy the following relations for all 
$\alpha \in \mathcal{O}(m)$,
$\beta \in \mathcal{O}(n)$,
$\gamma \in \mathcal{O}(p)$:
\begin{enumerate}
\item
sequential relation:
\[
\alpha \circ_i ( \beta \circ_j \gamma )
=
( \alpha \circ_i \beta ) \circ_{i-1+j} \gamma, 
\qquad 
1 \le i \le m, \, 1 \le j \le n;
\]
\item
parallel relation:
\[
( \alpha \circ_i \beta ) \circ_{j+n-1} \gamma
=
( \alpha \circ_j \gamma ) \circ_i \beta,
\qquad
1 \le i < j \le m;
\]
\item
unit relation:
\[
\mathrm{id} \circ_1 \alpha = \alpha = \alpha \circ_i \mathrm{id},
\qquad
1 \le i \le m.
\]
\end{enumerate}
The (nonzero) elements of $\mathcal{O}(n)$ are said to have \emph{arity} $n$.
\end{definition}

The nonzero elements of $\mathcal{O}(n)$ may be identified with linear combinations
of planar rooted trees with $n$ leaves.
The parallel and sequential relations can be understood 
in terms of grafting of trees:
$\alpha \circ_i \beta$ means that the root of $\beta$ is identified with the $i$-th leaf of $\alpha$.
Partial compositions can also be understood as representing insertion of the output of one operation
into an argument of another operation.
Interpret $\mathcal{O}(q)$ as the space of all $q$-ary operations on some (unspecified)
vector space.
If $\alpha \in \mathcal{O}(m)$ and $\beta \in \mathcal{O}(n)$
then we write $\alpha(x_1,\dots,x_m)$ and $\beta(y_1,\dots,y_n)$.
For $1 \le i \le m$ we have the $(m+n-1)$-ary operation
\[
( \alpha \circ_i \beta ) ( z_1, \dots, z_{m+n-1} )
=
\alpha( z_1, \dots, z_{i-1}, \beta( z_i, \dots, z_{i+n-1} ), z_{i+n}, \dots, z_{m+n-1} ).
\]

\begin{definition}
The operation $\omega \in \mathcal{O}(n)$ is \emph{associative} if every partial 
composition of $\omega$ with itself produces the same result in $\mathcal{O}(2n-1)$; 
that is, if
\[
\omega \circ_i \omega
=
\omega( x_1, \dots, x_{i-1}, \omega( x_i, \dots, x_{i+n-1} ), x_{i+n}, \dots, x_{2n-1} )
\]
does not depend on $i = 1, \dots, n$.
\end{definition}

Since we are concerned exclusively with nonsymmetric operads, for which the subscript 
of an argument represents its index in the sequence of arguments, without ambiguity
we may replace every argument by the generic argument symbol $\ast$.

\begin{definition}
Let $\na_{n,k}$ denote the free nonsymmetric operad generated by $k$ operations of arity $n$.
(We write $\na$ to represent \emph{non-associative}.)
Let $\na_{n,k}(p)$ denote the homogeneous subspace of arity $p$.
Clearly $\na_{n,k}(p) = \{0\}$ if $p \not\equiv 1$ (mod $n-1$).
The quantity $w = (p-1)/(n-1)$ is called the \emph{weight} of the (nonzero) elements
in $\na_{n,k}(p)$, which is the total number of occurrences of the generating operations 
in each basis monomial of $\na_{n,k}(p)$. 
Clearly $p = (n-1)w+1$.
\end{definition}

Figure \ref{namalgorithm} gives an algorithm for generating the monomial basis
of $\na_{n,k}(p)$ recursively by weight $w$.
The order in which the monomials are generated gives a total order on this monomial basis.
The $k$ operations are denoted by $\omega_i$ for $1 \le i \le k$.
Ordered sets (lists) are enclosed within square brackets.
The operation symbols are appended to the end of the lists representing monomials;
thus $x = [ x_1, \dots, x_n, \omega ]$ represents the composition of $n$ 
submonomials using the $n$-ary operation $\omega$.

\begin{figure}[ht]
\begin{itemize}
\item
\texttt{namonomials}[0] := $[ \,\ast\, ]$ (the list containing the generic argument symbol)
\item
for $w$ from 1 to \texttt{maxweight} do:
  \begin{enumerate}
  \item
  \texttt{namonomials}[$w$] := $[\;]$ (the empty list)
  \item
  for $c$ in compositions of $w{-}1$ into $n$ parts (allowing 0) do:
    \begin{enumerate}
    \item
    for $x \in \prod_{i=1}^n \texttt{namonomials}[c_i]$ (Cartesian product) do:
      \begin{enumerate}
      \item
      for $j = 1,\ldots, k$ do:
        \begin{itemize}
        \item[$\ast$]
        $x := [ x_1, \dots, x_n, \omega_j ]$
        \item[$\ast$]
        append $x$ to the list \texttt{namonomials}[$w$]
        \end{itemize}
      \end{enumerate}
    \end{enumerate}
  \end{enumerate}
\end{itemize}
\vspace{-8pt}
\caption{Algorithm for generating nonassociative monomials by weight}
\label{namalgorithm}
\end{figure}

For $k = 1$, the weight $w$ of a monomial $m$ is the number of occurrences of the single
$n$-ary operation $\omega_1$ in $m$.
The number of distinct monomials of weight $w$ 
is given by the Fuss-Catalan number
\[
\dim \na_{n,1}(p) = \frac{1}{p} \binom{nw}{w}.
\]
For $n = 3$ we obtain OEIS sequence A001764 (indexed by $w$ with offset 0):
\[
1, 1, 3, 12, 55, 273, 1428, 7752, 43263, 246675, 1430715, 8414640, 50067108, \dots
\]
For general $k$ there are $k$ choices for each $n$-ary operation, so we have
\[
\dim \na_{n,k}(p) = \frac{k^w}{p} \binom{nw}{w}.
\]


\subsubsection*{Rota-Baxter operads}

\begin{definition}
\label{operatormonomial}
Let $n \ge 2$ and $d \ge 1$.  
We consider \emph{operator monomials} $x$ generated by an indeterminate $\ast$ 
using an associative noncommutative $n$-ary multiplication
together with $d$ pairwise commuting unary operators $P_1$, $P_2$, $\dots$, $P_d$:
\begin{itemize}
\item
$\ast$ is a monomial;
\item
if $x_1$, $x_2$, \dots, $x_n$ are monomials then so is $x_1 x_2 \cdots x_n$
(using the associative $n$-ary multiplication);
\item
if $x$ is a monomial then so is $P_i(x)$ for $1 \le i \le d$.
\end{itemize}
The commutativity condition on the operators implies that for any submonomial 
of the form $P_i( P_j( x ) )$ we may assume that $i \le j$.
The number of occurrences of $\ast$ in $x$ is called the \emph{arity} of $x$.
\end{definition}

\begin{definition}
Let $\au_{n,d}$ denote the nonsymmetric operad generated by an associative $n$-ary operation 
$\omega$ and $d$ unary operators $P_1, \dots, P_d$ which commute in the sense that 
$P_i \circ_1 P_j = P_j \circ_1 P_i$ for $1 \le i < j \le d$.
(We write $\au$ to represent \emph{associative-unary}.)
This operad is multigraded where the homogeneous subspace $\au_{n,d}(p;\mathbf{m})$
of multidegree $(p,\mathbf{m})$ where $\mathbf{m} = (m_1,\dots,m_d)$
has a basis consisting of the operator monomials with arity $p$ 
and multiplicity $m_i$ of $P_i$ for $1 \le i \le d$.
\end{definition}

A dimension formula for $\au_{n,d}(p;\mathbf{m})$ is proved in \S\ref{dimensionformula}.
See the second column of Table \ref{dendriformrewriting} in \S\ref{ternarydendriform}
for the monomial basis of $\au_{3,1}(5;4)$.

\begin{definition}
We consider strings $w$ over the alphabet $A$ consisting of left and right parentheses $($ and $)$
together with $d$ pairs of indexed parentheses $(_i$ and $)_i$ for $1 \le i \le d$.
This gives a total order on $A$.
We call $w$ a (generalized) \emph{Dyck word} if 
\begin{itemize}
\item
every $($ is immediately followed by $)$; such a pair is called a \emph{nesting};
\item
the indexed parentheses are \emph{balanced}: For $1 \le i \le d$,
each initial substring contains at least as many $(_i$ as $)_i$ and the total numbers of $(_i$ and $)_i$ in $w$ are equal. 
\end{itemize}
The total order on Dyck words is induced by the total order on $A$.
Given two Dyck words $v = v_1 \dots v_k$ and $w = w_1 \cdots w_\ell$
where $v_i, w_j \in A$, we write $v \prec w$ if
\begin{itemize}
\item
either $k < \ell$,
\item
or $k = \ell$, $v_1 \cdots v_h = w_1 \cdots w_h$, $v_{h+1} \ne w_{h+1}$,
and $v_{h+1} \prec w_{h+1}$ ($h < \ell$).
\end{itemize}
\end{definition}

Dyck words are a convenient way of representing operator monomials.
We define an injective map $d$ from the set of operator monomials to the set of Dyck words:
\begin{itemize}
\item
$d( {\ast} ) := ( \, )$;
\item
if $x$ is an operator monomial then $d( P_i( x ) ) := (_i \, d(x) \, )_i$;
\item
if $x_1, \dots, x_n$ are operator monomials then $d( x_1 \cdots x_n ) := d(x_1) \cdots d(x_n)$.
\end{itemize}
The order on Dyck words induces an order on operator monomials.
The third column of Table \ref{dendriformrewriting} in \S\ref{ternarydendriform}
gives the Dyck words for the monomial basis of $\au_{3,1}(5;4)$.

\begin{definition}
\label{definitionrotabaxter}
Bai et al.~\cite[Definition 2.1(c)]{BGLW}.
Let $n \ge 2$, let $\omega$ be an $n$-ary operation (not necessarily associative)
denoted $\langle x_1, \dots, x_n \rangle$ on a vector space $V$ and let $\lambda \in \mathbb{F}$.
An \emph{$n$-ary Rota-Baxter operator of weight $\lambda$} 
is a unary operator $R = R_{n,\lambda}\colon V \to V$ which satisfies the relation
\[
\big\langle R(x_1), \dots, R(x_n) \big\rangle
\;
-
\sum_{\emptyset \ne I \subseteq [n]} 
\lambda^{|I|-1} 
R\left(
\big\langle 
\widehat{R}_I (x_1), \dots, \widehat{R}_I(x_n) 
\big\rangle
\right)
= 0,
\]
where
\[
\widehat{R}_I(x_i) :=
\begin{cases}
x_i &\text{if $i \in I$}, \\
R(x_i) &\text{if $i \notin I$}.
\end{cases}
\]
\end{definition}

When $\lambda \ne 0$, this relation is not homogeneous in $R$.
When $\lambda = 0$, only the terms with $|I| = 1$ appear, 
and we obtain the homogeneous relation
\[
\big\langle R(x_1), \dots, R(x_n) \big\rangle
-
\sum_{i=1}^n 
R \big( 
\big\langle R(x_1), \dots, R(x_{i-1}), x_i, R(x_{i+1}), \dots, R(x_n)  \big\rangle 
\big)
= 0.
\]
We will be concerned exclusively with RB operators of weight 0.
For $n = 2$ we obtain the relation of equation \eqref{binaryRB}.

\begin{remark}
Let $D$ be an $n$-ary derivation:
\[
D\left( \prod_{i=1}^n x_i \right)
=
\sum_{j=1}^n \prod_{i=1}^n D^{\delta_{ij}}( x_i ), \qquad
\delta_{ij} = \begin{cases} 1 &\text{if $i=j$} \\ 0 &\text{if $ i \ne j$} \end{cases}
\]
Assume that $D$ is invertible and let $E = D^{-1}$.
We claim that $E$ satisfies the $n$-ary RB relation of weight 0.
Let $x_i = E( y_i )$ for $1 \le i \le n$.
We have
\[
D\left( \prod_{i=1}^n E( y_i ) \right)
=
\sum_{j=1}^n \prod_{i=1}^n E^{1-\delta_{ij}}( y_i ).
\]
Applying $E$ to both sides proves the claim.
\end{remark}

\begin{definition}
\label{RBidealdefinition}
Let $\rb_{n,d}$ denote the nonsymmetric operad generated by an associative
$n$-ary operation $\omega$ and $d$ pairwise commuting unary operators 
$P_1, \dots, P_d$
each of which satisfies the $n$-ary RB relation of weight 0.
Since these RB relations are homogeneous in the operator,
this operad inherits the multigrading of $\au_{n,d}$.
We regard $\rb_{n,d}$ as the quotient of $\au_{n,d}$ modulo 
the ideal $\mathcal{I}_{n,d}$ generated by the RB relations.
We write $s_{n,d} \colon \au_{n,d} \to \au_{n,d} / \mathcal{I}_{n,d}$
for the natural surjection.
\end{definition}

\begin{openproblem}
Find a closed formula for $\dim \rb_{n,d}(p,m_1,\dots,m_d)$
where $p$ is the arity and $m_i$ is the multiplicity of $P_i$ for $1 \le i \le d$.
This dimension is the number of basis monomials for this homogeneous subspace.
For the case $n = 2$, $d = 1$ see Guo and Sit \cite{GuoSit}.
\end{openproblem}


\subsubsection*{Consequences of the Rota-Baxter relation}

In order to find spanning sets for the homogeneous subspaces of the ideal $\mathcal{I}_{n,d}$,
we need an algorithm for generating the consequences of the RB relations with
higher arity and operator multiplicities.
This algorithm is displayed in Figure \ref{generatingconsequences}.

\begin{figure}[ht]
\begin{enumerate}
\item
$\texttt{permlist} := \text{all permutations of the multiset $\mathcal{M}_i$}$ (see text)
\item
For $k = 1, \dots, \text{size}(\mathcal{M}_i)$ do:
  \begin{enumerate}
  \item
  $\texttt{perm} := \texttt{permlist}[k]$
  \item
  $\texttt{relations} := [ \, R_i \, ]$ \; \text{the RB relation for $P_i$}
  \item
  $\texttt{arity} := n$
  \item
  For $p$ in $\texttt{perm}$ do:
    \begin{enumerate}
    \item
    $\texttt{temporary} := [\,]$ (empty list)
    \item
    If $p = A$ then
      \begin{itemize}
      \item
      For $r$ in $\texttt{relations}$ do
        \begin{itemize}
        \item
        For $j = 1,\dots,n$ append $A \circ_j r$ to $\texttt{temporary}$
        \item
        For $j = 1, \dots, \texttt{arity}$ append $r \circ_j A$ to $\texttt{temporary}$
        \end{itemize}
      \item
      $\texttt{arity} := \texttt{arity} + (n-1)$
      \end{itemize}
    \item
    If $p = P_\ell$ for some $\ell = 1,\dots,d$ then
      \begin{itemize}
      \item
      For $r$ in $\texttt{relations}$ do
        \begin{itemize}
        \item
        Append $P_\ell \circ_1 r$ to $\texttt{temporary}$
        \item
        For $j = 1, \dots, \texttt{arity}$ append $r \circ_j P_\ell$ to $\texttt{temporary}$
        \end{itemize}
      \end{itemize}
    \item
    $\texttt{relations} := \texttt{temporary}$
    \end{enumerate}
  \item
  $\texttt{consequences}[i] := \texttt{relations}$
  \end{enumerate}
\end{enumerate}
\vspace{-8pt}
\caption{Algorithm to generate consequences of the RB relation $R_i$}
\label{generatingconsequences}
\end{figure}

The RB relation $R_i$ for $P_i$ belongs to $\au_{n,d}(n;n\mathbf{e}_i)$ 
where $\mathbf{e}_i$ is the $i$-th standard unit vector
(weight 1, multiplicity $n$ in position $i$).
We wish to generate the consequences of $R_1, \dots, R_d$
in $\au_{n,d}(p;\mathbf{m})$ where $p = 1 + k(n-1)$ (weight $k$) and $\mathbf{m} = ( m_1, \dots, m_d )$.
For operator $P_i$, we need to apply partial composition $k-1$ times with $A$
(increasing the weight from 1 to $k$), 
and for each $j = 1, \dots, d$ we need to apply partial composition $m_j - n\delta_{ij}$ times with $P_j$.
We write $\mathcal{M}_i$ for the multiset consisting of $k-1$ copies of $A$
and $m_j - n\delta_{ij}$ copies of $P_j$ for $j = 1, \dots, d$.

\begin{example}
\label{exampledendriformconsequences}
Suppose that $n = 3$ and $d = 1$, and write $P_1 = S$.
The ternary RB relation can be written as $R = 0$ where
\[
R := 
S({\ast})S({\ast})S({\ast}) - S({\ast}S({\ast})S({\ast})) 
- S(S({\ast}){\ast}S({\ast})) - S(S({\ast})S({\ast}){\ast}).
\]
This relation has arity 3 and operator multiplicity 3.
In \S\ref{ternarydendriform} we require the consequences of $R$ with arity $p = 5$, 
multiplicity $m_1 = 4$, 
which lie in the codomain of the restriction of the rewriting morphism
to the quadratic subspace of the domain.
Thus we must apply the associative operation once and the unary operator once;
hence $\mathcal{M} = \{ A, S \}$.
We obtain these consequences in two ways, starting with $R$:
either partial composition with $A$ and then with $S$,
or the reverse.
For permutation $[ A, S ]$ (respectively $[ S, A ]$) we obtain 36 (respectively 24)
consequences; altogether 42 distinct consequences.
We list two consequences each for $[A,S]$ and $[S,A]$:
\begin{align*}
S \circ_1 ( A \circ_1 R ) 
&=
S({\ast}{\ast}S({\ast})S({\ast})S({\ast})) - S({\ast}{\ast}S({\ast}S({\ast})S({\ast})))
\\
&\quad
{}
- S({\ast}{\ast}S(S({\ast}){\ast}S({\ast}))) - S({\ast}{\ast}S(S({\ast})S({\ast}){\ast})),
\\
( R \circ_1 A ) \circ_1 S
&=
S(S({\ast}){\ast}{\ast})S({\ast})S({\ast}) - S(S({\ast}){\ast}{\ast}S({\ast})S({\ast})) 
\\
&\quad
{}- S(S(S({\ast}){\ast}{\ast}){\ast}S({\ast})) - S(S(S({\ast}){\ast}{\ast})S({\ast}){\ast}),
\\
A \circ_1 ( S \circ_1 R ) 
&=
S(S({\ast})S({\ast})S({\ast})){\ast}{\ast} - S(S({\ast}S({\ast})S({\ast}))){\ast}{\ast} 
\\
&\quad
{}
- S(S(S({\ast}){\ast}S({\ast}))){\ast}{\ast} - S(S(S({\ast})S({\ast}){\ast})){\ast}{\ast},
\\
( R \circ_1 S ) \circ_1 A
&=
S(S({\ast}{\ast}{\ast}))S({\ast})S({\ast}) - S(S({\ast}{\ast}{\ast})S({\ast})S({\ast})) 
\\
&\quad
{}
- S(S(S({\ast}{\ast}{\ast})){\ast}S({\ast})) - S(S(S({\ast}{\ast}{\ast}))S({\ast}){\ast}).
\end{align*}
\end{example}

\begin{example}
\label{multisetexample}
Suppose $n = 3$, $d = 2$, and write $P_1 = T$, $P_2 = S$.
In \S\ref{ternaryquadrialgebras} we require the consequences of $R$ with $p = 5$ and $(m_1,m_2) = (4,4)$.
Hence $\mathcal{M}_1 = \{ A, S, S, S, S, T \}$ and $\mathcal{M}_2 = \{ A, S, T, T, T, T \}$;
each multiset has 30 permutations.
\end{example}


\subsubsection*{The rewriting morphism}

\begin{definition}
We consider the general case of an $n$-ary operation together
with $d$ commuting unary operators.
We define the \emph{rewriting morphism} $r_{n,d} \colon \na_{n,k} \to \au_{n,d}$ where $k = n^d$.
Since $\na_{n,k}$ is a free operad, it suffices to define $r_{n,d}$ on the $n^d$ generating
operations: the monomial basis of $\na_{n,k}(n)$ denoted by
$\omega_{\ell_1,\dots,\ell_d}$
where
$1 \le \ell_1,\dots,\ell_d \le n$.
The image of $\omega_{\ell_1,\dots,\ell_d}$ under $r_{n,d}$ is an operator monomial in 
$\au_{n,d}(n;\mathbf{m})$ where $\mathbf{m} = ( n{-}1, \dots, n{-}1 )$.
We define $r_{n,d}( \omega_{\ell_1,\dots,\ell_d} )$ to be the monomial in which
operator $P_i$ ($1 \le i \le d$) is applied to each argument \emph{except} $\ell_i$:
\[
r_{n,d}( \omega_{\ell_1,\dots,\ell_d} )
:=
\prod_{i=1}^n \left( \prod_{j=1}^d P_j^{\epsilon_{ij}} \right) ( x_i ),
\qquad
\epsilon_{ij} 
=
\begin{cases}
\; 0 &\text{if $i = \ell_j$;} \\
\; 1 &\text{otherwise.}
\end{cases}
\]
\end{definition}

\begin{definition}
\label{compositiondefinition}
The composition $f_{n,d} := s_{n,d} \circ r_{n,d}$ of the rewriting morphism
with the surjection is an operad morphism from 
the free operad $\na_{n,k}$ ($k = n^d$) to the Rota-Baxter operad $\rb_{n,d}$.
The kernel of $f_{n,d}$ contains all the relations satisfied by the rewritings in
$\au_{n,d}$ of the generating operations of the free operad $\na_{n,k}$;
that is, the nonassociative relations implied by the Rota-Baxter relation(s).
In what follows our computations will be exclusively concerned with the restriction of $f_{n,d}$ 
to the quadratic (weight 2) subspace of its domain.
\end{definition}

\begin{example}
\label{rewritingternarydendriform}
If $n = 3$ and $d = 1$ we write $S = P_1$ and obtain
\[
r_{3,1}( \omega_1 ) = x S(y) S(z),
\qquad
r_{3,1}( \omega_2 ) = S(x) y S(z),
\qquad
r_{3,1}( \omega_3 ) = S(x) S(y) z.
\]
\end{example}

\begin{example}
\label{rewritingquadrialgebras}
If $n = 3$ and $d = 2$ we write $T = P_1$, $S = P_2$, $r = r_{3,2}$ and obtain
\begin{alignat*}{3}
r( \omega_{11} ) &= x TS(y) TS(z),
&\quad
r( \omega_{12} ) &= T(x) S(y) TS(z),
&\quad
r( \omega_{13} ) &= T(x) TS(y) S(z),
\\
r( \omega_{21} ) &= S(x) T(y) TS(z),
&\quad
r( \omega_{22} ) &= TS(x) y TS(z),
&\quad
r( \omega_{23} ) &= TS(x) T(y) S(z),
\\
r( \omega_{31} ) &= S(x) TS(y) T(z),
&\quad
r( \omega_{32} ) &= TS(x) S(y) T(z),
&\quad
r( \omega_{33} ) &= TS(x) TS(y) z.
\end{alignat*}
\end{example}

For $d = 3$ commuting operators we obtain 27 ternary operations
whose quadratic relations define the ternary analog of the octo-algebras of Leroux \cite{Leroux}.


\section{A dimension formula for $\au_{n,d}$}
\label{dimensionformula}

The results of this section form part of a research program on the combinatorics of
multi-operator monomials 
for $n$-ary algebras; see two recent preprints by the authors \cite{AB1,AB2}.
For related work on operated semigroups, see Guo \cite{Guo1}; see also
Bokut et al.~\cite{BCQ} and Liu et al.~\cite{LQQZ} for associative algebras
with multiple operators.

For an operator monomial $x$ (Definition~\ref{operatormonomial}), let
$r=\deg(x)$ 
be its arity, that is, the number of occurrences of the argument symbol $\ast$. Also, let
\[
\mathbf{m}=(m_1,\dots,m_d)\in\mathbb{Z}_{\geq 0}^d
\]
record the multiplicities of the unary operators $P_1,\dots,P_d$.
We write
\[
a_{n,d}(r;\mathbf{m})
=
\dim \au_{n,d}(r;\mathbf{m})
\]
for the number of operator monomials of arity $r$ and multiplicity vector
$\mathbf{m}$.

We count these monomials in two stages.  First, we record the associative
multiplication structure by an ordered rooted tree.  Second, we distribute
canonically ordered unary-operator chains among the vertices of this tree.
Associativity imposes a condition on these trees: every non-root internal
multiplication vertex must be separated from its parent by a nonempty unary chain.

By a \emph{unary chain} we mean a sequence of unary operators
$P_{i_1}P_{i_2}\cdots P_{i_k}$
applied consecutively to a submonomial.  A unary chain is \emph{maximal} if it is
not properly contained in a longer chain of unary operators.
Since the operators commute, every chain has a unique canonical
representative in which
$i_1\leq i_2\leq\cdots\leq i_k$.
Thus a unary chain is determined by the multiplicities of the operators occurring
in it.


\subsubsection*{Multiplication skeletons}

Begin with the expression tree of an operator monomial.  Decorate each
edge by the maximal unary chain occurring on that edge, using the empty
chain when no unary operator occurs there.  Next, contract every edge
joining a non-root multiplication vertex to its parent multiplication
vertex whenever that edge has the empty decoration. Associativity ensures that
this contraction does not change the represented operator monomial.  After the
contractions, temporarily forget the unary decorations.  The resulting ordered
rooted tree is called the \emph{multiplication skeleton}.  Its leaves represent
occurrences of $\ast$, and its internal vertices represent maximal associative
products.
Every internal vertex has outdegree
$1+(n-1)j$ for some $j\geq 1$,
because a composite containing $j$ occurrences of an $n$-ary multiplication has
$1+(n-1)j$ inputs.

Let $T_n(r,I)$ denote the number of multiplication skeletons with $r$ leaves and
$I$ internal multiplication vertices.  Necessarily,
$r\equiv 1 \pmod{n-1}$.
When this congruence holds, put
$w=(r-1)/(n-1)$.
Thus $w$ is the weight: before associativity is imposed, it is the number of
occurrences of the generating $n$-ary multiplication.  After adjacent
multiplication vertices are collapsed, the multiplication skeleton may have fewer
internal vertices, so
$1\leq I\leq w$.
The exceptional case $r=1$ consists of a single leaf and has $I=0$.

\begin{lemma}
\label{lemmaT}
Let $r>1$.  If $r\not\equiv 1\pmod{n-1}$, then $T_n(r,I)=0$.
Suppose that
\[
w=\frac{r-1}{n-1}\in\mathbb{Z}_{>0}.
\]
Then
\[
T_n(r,I)
=
\frac{1}{r+I}
\binom{r+I}{r}
\binom{w-1}{I-1},
\qquad
1\leq I\leq w,
\]
and $T_n(r,I)=0$ for all other values of $I$.
\end{lemma}

\begin{proof}
Set
$N=r+I$.
We use the standard correspondence between ordered rooted trees and {\L}ukasiewicz words; 
see, for example, \cite[\S5.3]{StanleyEC2}. 
Traverse a multiplication skeleton in preorder, and let $d_k$ be the outdegree of its $k$-th vertex.
The sequence 
$(d_1-1,\dots,d_N-1)$
is called its {\L}ukasiewicz word. It has total sum $-1$, and every proper initial sum is nonnegative. Conversely, every integer sequence with these properties is the preorder outdegree sequence of a unique ordered rooted tree.

In a multiplication skeleton, a leaf has outdegree $0$, and an internal vertex has outdegree 
$1+(n-1)j$
for some $j \ge 1$. Thus the {\L}ukasiewicz 
sequence contains $r$ entries equal to $-1$ and $I$ positive entries of the form
\[
(n-1)j_1, \, \dots, \, (n-1)j_I,
\qquad
j_1, \dots, j_I\geq 1.
\]
Since the total sum is $-1$, we have
\[
-r+(n-1)(j_1+\cdots+j_I)=-1.
\]
Therefore
\[
j_1+\cdots+j_I
=
\frac{r-1}{n-1}
=
w.
\]
The number of ordered compositions of $w$ into $I$ positive parts is
$\binom{w-1}{I-1}$.
For each such composition, the $r$ entries equal to $-1$ may be placed
among the $r+I$ positions in
$\binom{r+I}{r}$
ways. Hence the total number of unrestricted sequences is
\[
\binom{r+I}{r}
\binom{w-1}{I-1}.
\]
It remains to determine which cyclic shifts of these unrestricted
sequences are {\L}ukasiewicz words.  Fix one such sequence
$(a_1,\dots,a_N)$ with $N=r+I$,
and define its partial sums by
$s_0=0$ and 
$s_k=a_1+\cdots+a_k$
for
$1\leq k\leq N$.
Since the entries have total sum $-1$, we have $s_N=-1$. Let $j$ be the first index at which the minimum of
$s_0,s_1,\dots,s_{N-1}$
is attained.  We claim that the cyclic shift
$(a_{j+1},\dots,a_N,a_1,\dots,a_j)$
has all proper initial sums nonnegative. For an initial segment that does not pass through the end of the original sequence, its sum has the form
$s_k-s_j$
for some $j<k\leq N-1$, and this is nonnegative by the choice of $j$.
For an initial segment that passes through the end of the original
sequence, its sum has the form
$(s_N-s_j)+s_k
=
-1-s_j+s_k$
for some $0\leq k<j$.  Since $j$ is the first index at which the
minimum is attained, we have $s_k>s_j$.  The partial sums are integers,
so $s_k\geq s_j+1$, and hence
$-1-s_j+s_k\geq0$.
Thus this cyclic shift is a {\L}ukasiewicz word. Conversely, suppose that the cyclic shift beginning with $a_{j+1}$ has all proper initial sums nonnegative.  The initial sums before
wrapping around imply 
$s_k\geq s_j$
where
$j<k\leq N-1$,
while those after wrapping around imply
$-1-s_j+s_k\geq0$
for $0\leq k<j$.
Thus $s_k>s_j$ for $k<j$ and $s_k\geq s_j$ for $k>j$, so $j$ is the
first index at which the minimum of
$s_0,\dots,s_{N-1}$ is attained.  Therefore the valid cyclic shift is
unique.

Moreover, the sequence cannot have a nontrivial period.  Indeed, if it
were obtained by repeating a shorter block $q>1$ times, then its total
sum would be divisible by $q$, which is impossible because its total
sum is $-1$.  Thus every cyclic orbit has exactly $r+I$ elements.
Consequently, every cyclic orbit of unrestricted sequences has
$r+I$ elements and contains exactly one {\L}ukasiewicz word.  Dividing
the number of unrestricted sequences by $r+I$ proves the claim.
\end{proof}


\subsubsection*{Adding commuting unary operators}

Now fix a multiplication skeleton with $r$ leaves and $I$ internal multiplication
vertices.  It has $r+I$ vertices in total.  We regard each vertex as a possible
location for a maximal unary chain, attached immediately above that vertex.
The chain at the root may be empty, as may the chains at the leaves.  Every
non-root internal multiplication vertex, however, must carry a nonempty unary
chain.  Otherwise that vertex would be adjacent to its parent multiplication
vertex with no unary operator separating them, and associativity would collapse
the two vertices into a single multiplication vertex.  Thus the $I-1$ non-root
internal vertices are constrained to carry nonempty chains, while the root and
the $r$ leaves give $r+1$ unconstrained locations.

For a fixed skeleton, let $U_d(r,I;\mathbf{m})$ denote the number of ways to
distribute the commuting unary operators with total multiplicity vector
$\mathbf{m}$ among the vertices of the skeleton, subject to the condition that
every non-root internal multiplication vertex receives a nonempty chain. 
We compute this number by inclusion--exclusion.

\begin{lemma}
\label{lemmaU}
For a fixed multiplication skeleton with $r$ leaves and $I$ internal
vertices, the number of admissible unary-chain assignments with total
multiplicity vector
$\mathbf{m}=(m_1,\dots,m_d)$
is
\[
U_d(r,I;\mathbf{m})
=
\sum_{k=0}^{I-1}
(-1)^k
\binom{I-1}{k}
\prod_{j=1}^d
\binom{m_j+r+I-k-1}{r+I-k-1}.
\]
\end{lemma}

\begin{proof}
Temporarily label the $I-1$ non-root internal vertices
$v_1,\dots,v_{I-1}$,
and let $E_h$ be the event that the unary chain at $v_h$ is empty.
The final count is independent of the chosen labeling. 
Suppose that a specified set of $k$ constrained vertices is required
to carry empty chains.  The unary operators may then be distributed
among the remaining
$r+I-k$
vertices.

For each $j=1,\dots,d$, distributing the $m_j$ indistinguishable
copies of $P_j$ among these $r+I-k$ vertices is equivalent to choosing
a weak composition
\[
m_j=b_1+\cdots+b_{r+I-k},
\qquad
b_\ell\geq0.
\]
The number of such weak compositions is
\[
\binom{m_j+r+I-k-1}{r+I-k-1}.
\]
The choices for the different operator types are independent.
Moreover, since each unary chain is written in canonical order, the
numbers of copies of $P_1,\dots,P_d$ assigned to a vertex uniquely
determine the chain at that vertex.
There are
$\binom{I-1}{k}$
ways to choose the $k$ constrained vertices that carry empty chains.
Inclusion--exclusion over the events
$E_1,\dots,E_{I-1}$ therefore gives the stated formula.
\end{proof}

\begin{theorem}
\label{theoremenumeration}
Let $n\geq 2$, $d\geq 1$, and let
$\mathbf{m}=(m_1,\dots,m_d)\in\mathbb{Z}_{\geq 0}^d$.
If $r=1$ then
$a_{n,d}(1;\mathbf{m})=1$.
If $r>1$ and
$r\not\equiv 1\pmod{n-1}$
then
$a_{n,d}(r;\mathbf{m})=0$.
Suppose that $r>1$ and
$w=(r-1)/(n-1)\in\mathbb{Z}_{>0}$.
Then
\begin{align*}
&
a_{n,d}(r;\mathbf{m})
=
\\
&
\sum_{I=1}^{w}
\frac{1}{r+I}
\binom{r+I}{r}
\binom{w-1}{I-1}
\sum_{k=0}^{I-1}
(-1)^k
\binom{I-1}{k}
\prod_{j=1}^d
\binom{m_j+r+I-k-1}{r+I-k-1}.
\end{align*}
\end{theorem}
\begin{proof}
If $r=1$, the unique operator monomial is
$P_1^{m_1}\cdots P_d^{m_d}(\ast)$,
where the operators are written in canonical order.

Now suppose that $r>1$.  Every application of an $n$-ary
multiplication increases the arity by $n-1$, so no monomial of arity
$r$ exists unless
$r\equiv 1\pmod{n-1}$.
Assume that
$w=(r-1)/(n-1)$
is a positive integer.  For each $I$ with $1\leq I\leq w$, there are
$T_n(r,I)$ multiplication skeletons with $r$ leaves and $I$ internal
vertices.  For each such skeleton, there are
$U_d(r,I;\mathbf{m})$ admissible assignments of unary chains with
total multiplicity vector $\mathbf{m}$.  Therefore
\[
a_{n,d}(r;\mathbf{m})
=
\sum_{I=1}^{w}
T_n(r,I)\,U_d(r,I;\mathbf{m}).
\]
Substituting the formulas from Lemmas~\ref{lemmaT} and \ref{lemmaU} gives the stated result.
\end{proof}

We also describe a compact coefficient-extraction form of Theorem~\ref{theoremenumeration}.

\begin{corollary}
\label{corollarycoefficientform}
Let
\[
\mathbf{x}^{\mathbf{m}}
=
x_1^{m_1}\cdots x_d^{m_d},
\qquad
A(\mathbf{x})
=
\prod_{j=1}^d\frac{1}{1-x_j}.
\]
If $r>1$ and
$w=\frac{r-1}{n-1}\in\mathbb{Z}_{>0}$,
then
\[
a_{n,d}(r;\mathbf{m})
=
[\mathbf{x}^{\mathbf{m}}]\,
A(\mathbf{x})^{r+1}
\sum_{I=1}^{w}
T_n(r,I)
\bigl(A(\mathbf{x})-1\bigr)^{I-1}.
\]
\end{corollary}

\begin{proof}
The series $A(\mathbf{x})$ is the generating function for an arbitrary
commuting unary chain, including the empty chain, while
$A(\mathbf{x})-1$ is the generating function for a nonempty commuting
unary chain.
A multiplication skeleton with $r$ leaves and $I$ internal vertices
has $r+1$ unconstrained vertices and $I-1$ constrained vertices.
Consequently, its unary-chain contribution is
\[
[\mathbf{x}^{\mathbf{m}}]\,
A(\mathbf{x})^{r+1}
\bigl(A(\mathbf{x})-1\bigr)^{I-1}.
\]
Multiplying by $T_n(r,I)$ and summing over $I$ proves the result.
\end{proof}

\begin{example}
\label{dimauexample1}
Let $n=3$ and $r=5$.  Then $w=2$, and Lemma~\ref{lemmaT} gives
$T_3(5,1)=1$ and $T_3(5,2)=3$.
Thus
\[
a_{3,d}(5;\mathbf{m})
=
[\mathbf{x}^{\mathbf{m}}]
\left(
A(\mathbf{x})^6
+3A(\mathbf{x})^6\bigl(A(\mathbf{x})-1\bigr)
\right)
=
[\mathbf{x}^{\mathbf{m}}]
\left(3A(\mathbf{x})^7-2A(\mathbf{x})^6\right).
\]
For one unary operator,
\[
a_{3,1}(5;m)
=
3\binom{m+6}{6}-2\binom{m+5}{5},
\]
and hence
$a_{3,1}(5;4) = 378$.
For two commuting unary operators,
\[
a_{3,2}(5;m_1,m_2)
=
3\prod_{j=1}^2\binom{m_j+6}{6}
-
2\prod_{j=1}^2\binom{m_j+5}{5},
\]
and hence
$a_{3,2}(5;4,4)
=
100548$.
\end{example}

\begin{example}
\label{dimauexample2}
The following symmetric array gives $\dim \au_{3,2}(5;m_1,m_2)$:
\[
\begin{array}{l|rrrrrrr}
{}_{m_1}\backslash^{m_2} & 0 & 1 & 2 & 3 & 4 & 5 & 6 \\ \midrule
0 & 1 & 9 & 42 & 140 & 378 & 882 & 1848 \\
1 & 9 & 75 & 336 & 1092 & 2898 & 6678 & 13860 \\
2 & 42 & 336 & 1470 & 4704 & 12348 & 28224 & 58212 \\
3 & 140 & 1092 & 4704 & 14896 & 38808 & 88200 & 181104 \\
4 & \boxed{378} & 2898 & 12348 & 38808 & \boxed{100548} & 227556 & 465696 \\
5 & 882 & 6678 & 28224 & 88200 & 227556 & 513324 & 1047816 \\
6 & 1848 & 13860 & 58212 & 181104 & 465696 & 1047816 & 2134440
\end{array}
\]
The boxed values
\[
\dim \au_{3,2}(5;4,0) = 378 = \dim \au_{3,1}(5;4),
\qquad
\dim \au_{3,2}(5;4,4)=100548,
\]
are the dimensions of the codomains of the
restrictions of the rewriting morphisms to the quadratic subspaces in \S\ref{ternarydendriform} and
\S\ref{ternaryquadrialgebras}, respectively.  The first row of this table appears as
OEIS A051923, with formula $\frac{n+2}{2}\binom{n+5}{5}$; these numbers are the
partial sums of A051836, the 5-dimensional version of the pentagonal-based
pyramidal numbers.  In this row $m_1=0$, so only one operator is present.  At the
time of writing, the other rows were not in the OEIS \cite{OEIS}.
\end{example}


\section{Two propositions on $n$-ary Rota-Baxter operads}
\label{twopropositions}

Recall the $n$-ary RB relation of weight 0 from Definition \ref{definitionrotabaxter},
which can be written as follows assuming that the $n$-ary operation is associative:
\[
\prod_{i=1}^n R(x_i) 
=
\sum_{i=1}^n R \left( R(x_1) \cdots \widehat{R}(x_i) \cdots R(x_n) \right),
\]
where $\widehat{R}$ indicates that the operator is missing from this factor.
This suggests the following $n$ new $n$-ary operations, by analogy with the binary case:
\begin{equation}
\label{noperations}
X_i := R(x_1) \cdots \widehat{R}(x_i) \cdots R(x_n) \qquad (1 \le i \le n).
\end{equation}

\begin{proposition}
\label{proposition1}
Let $R$ be the binary Rota-Baxter operator in the operad $\rb_{2,1}$.
Then $R$ also satisfies the $n$-ary RB relation for all $n \ge 2$.
This gives a morphism of operads $f_n \colon \rb_{n,1} \to \rb_{2,1}$.
\end{proposition}

\begin{proof}
The associative operation in the binary RB operad 
gives rise to an $n$-ary associative operation in the obvious way.
Assume that $R$ satisfies the binary RB relation.
We show that for $n \ge 2$, if $R$ satisfies the $n$-ary RB relation
then $R$ also satisfies the $(n+1)$-ary RB relation.
The left side of the $(n+1)$-ary RB relation is
\[
\prod_{i=1}^{n+1} R(x_i) 
=
\left( \prod_{i=1}^n R(x_i) \right) R(x_{n+1}).
\]
Applying the $n$-ary RB relation we obtain
\[
\sum_{i=1}^n R \left( R(x_1) \cdots \widehat{R}(x_i) \cdots R(x_n) \right) R(x_{n+1}).
\]
Each term is a product $R(X_i)R(x_{n+1})$ where
$X_i = R(x_1) \cdots \widehat{R}(x_i) \cdots R(x_n)$,
so we may apply the binary RB relation to the last expression to obtain
\begin{align*}
&
\sum_{i=1}^n 
R \Big( R(X_i) x_{n+1} \Big) 
+ 
\sum_{i=1}^n 
R \Big( X_i R(x_{n+1}) \Big)
\\
&=
R \Big( \sum_{i=1}^n R(X_i) x_{n+1} \Big) 
+ 
\sum_{i=1}^n 
R \Big( X_i R(x_{n+1}) \Big).
\end{align*}
By the $n$-ary RB relation we have
\[
\sum_{i=1}^n 
R( X_i ) 
=
\prod_{i=1}^n R(x_i).
\]
We now obtain
\[
R \Big( \prod_{i=1}^n R(x_i) x_{n+1} \Big)
+ 
\sum_{i=1}^n 
R \Big( X_i R(x_{n+1}) \Big)
\]
This is the right side of the $(n+1)$-ary RB relation.
\end{proof}

\begin{conjecture}
The operad morphism $f_n$ in Proposition \ref{proposition1} is injective.
Hence $\rb_{2,1}$ has a suboperad isomorphic to $\rb_{n,1}$ for all $n \ge 3$.
\end{conjecture}

In the $n$-ary RB operad define a new operation as the sum of the $n$ operations
from equation \eqref{noperations}; that is
$( x_1, \dots, x_n ) := \sum_{i=1}^n X_i$.

\begin{proposition}
\label{sumoperationassociative}
The operation $( x_1, \dots, x_n )$ is an associative $n$-ary operation.
\end{proposition}

\begin{proof}
We verify that each of the $n$ association types in arity $2n-1$ 
for the sum operation reduces to the same RB monomial, 
which is the right side of the $(2n{-}1)$-ary RB relation.
Consider the $i$-th association type for the sum operation:
\[
( x_1, \dots, x_{i-1}, ( x_i, \dots, x_{i+n-1} ), x_{i+n}, \dots, x_{2n-1} ),
\qquad
1 \le i \le n.
\]
Applying the definition of the sum operation to the inner product gives
\[
( x_1, \dots, x_{i-1}, 
\sum_{j=1}^n R(x_i) \cdots \widehat{R}(x_{i-1+j}) \cdots R(x_{i-1+n}), 
x_{i+n}, \dots, x_{2n-1} ).
\]
For convenience, write
\[
\Sigma_i = \sum_{j=1}^n R(x_i) \cdots \widehat{R}(x_{i-1+j}) \cdots R(x_{i-1+n}).
\]
Now apply the definition of the sum operation again to obtain
\begin{align*}
&
\sum_{k=1}^{i-1} 
R(x_1) \cdots \widehat{R}(x_k) \cdots R(x_{i-1})
R( \Sigma_i ) 
R(x_{i+n}) \cdots R(x_{2n-1})
\\
&
+
R(x_1) \cdots R(x_{i-1})
\,\Sigma_i\, 
R(x_{i+n}) \cdots R(x_{2n-1})
\\
&
+
\sum_{k=i+n}^{2n-1}
R(x_1) \cdots R(x_{i-1})
R( \Sigma_i ) 
R(x_{i+n}) \cdots \widehat{R}(x_k) \cdots R(x_{2n-1}).
\end{align*}
By the $n$-ary RB relation we have
$R( \Sigma_i ) = R(x_i) \cdots R(x_{i-1+n})$.
We now obtain
\[
\sum_{k=1}^{2n-1}
R(x_1) \cdots \widehat{R}(x_k) \cdots R(x_{2n-1}),
\]
which does not depend on $i$.
\end{proof}


\section{The ternary analog of dendriform algebras}
\label{ternarydendriform}


\subsubsection*{Domain of the restriction of the rewriting morphism}

The operad $\na_{3,3}$ is generated by nonassociative ternary operations $\omega_i$ for $1 \le i \le 3$.
We consider the restriction of the rewriting morphism to the homogeneous subspace $\na_{3,3}(5)$, 
which has a basis of quadratic (weight 2) monomials 
$x_{ijk} = \omega_i \circ_j \omega_k$ for $1 \le i, j, k \le 3$.
This homogeneous component is the space of formal quadratic expressions in the
generating operations. 
We order these monomials by lex order on the triples $(j,k,i)$.
In terms of nonassociative monomials, we sort
first by association type (placement of parentheses), 
then by the inner operation, and finally by the outer operation:
\[
(({\ast}{\ast}{\ast})_k{\ast}{\ast})_i \;\; (j = 1),
\qquad
({\ast}({\ast}{\ast}{\ast})_k{\ast})_i \;\; (j = 2),
\qquad
({\ast}{\ast}({\ast}{\ast}{\ast})_k)_i \;\; (j = 3).
\]
See the first column of Table \ref{dendriformrewriting}
for the 27 basis monomials of $\na_{3,3}(5)$ in lex order.


\subsubsection*{Codomain of the restriction of the rewriting morphism}

The operad $\au_{3,1}$ is generated by an associative ternary operation and a unary operator $S$.
The codomain of the restriction is the homogeneous subspace $\au_{3,1}(5;4)$.
A basis of this subspace consists of the 378 operator monomials of arity 5 and operator multiplicity 4;
see Examples \ref{dimauexample1} and \ref{dimauexample2}.
The second column of Table \ref{dendriformrewriting}
lists the operator monomials which occur in the image of the restriction.

Figure \ref{operatormonomials} gives an algorithm which generates all $n$-ary 
one-operator monomials up to a given arity and multiplicity.
Each monomial is represented as a list (an ordered set enclosed in square brackets) 
and brackets are also used without ambiguity to indicate the application of the unary operator.
Observe that any monomial of operator multiplicity $m \ge 1$
can be obtained from a monomial of multiplicity $m-1$ by enclosing 
some valid submonomial in brackets, where valid means that the sublist 
has length congruent to 1 modulo $n-1$.

\begin{figure}[ht]
\begin{itemize}
\item
$n$ := the arity of the associative operation
\item
\texttt{maxp} := maximum arity of operator monomials to be generated
\item[]
(\texttt{maxp} must be congruent to 1 modulo $n-1$.)
\item
\texttt{maxm} := maximum multiplicity of operator monomials to be generated
\item
for $p$ from 1 to \texttt{maxp} in steps of $n{-}1$ do
($p$ is the arity of the monomials)
\begin{enumerate}
\item
(multiplicity 0: only monomial is list of $p$ argument symbols $\ast$)
\item[]
\texttt{monomials}[ $p$, 0 ] := [ [ $\ast$, \dots, $\ast$ ] ]
\item
for $m$ from $1$ to \texttt{maxm} do ($m$ is the multiplicity of the monomials)
\begin{enumerate}
\item
\texttt{monomials}$[p,m]$ := $[\;]$ \, (the empty list)
\item
(loop through all monomials with one less operator symbol)
\item[]
for $x$ in \texttt{monomials}$[ p, m{-}1 ]$ do
\begin{enumerate}
\item
$k$ := length($x$) (the number of items in the list $x$)
\item
(double loop through all submonomials of $x$)
\item[]
for $i$ from 1 to $k$ do and for $j$ from $i$ to $k$ do
\begin{itemize}
\item[$\ast$]
if $j{-}i{+}1 \equiv 1$ modulo $n{-}1$ then
\begin{itemize}
\item[$\bullet$]
(add operator brackets around the valid submonomial)
\item[]
$x'$ := [ $x_1$, \dots, $x_{i-1}$, [ $x_i$, \dots, $x_j$ ], $x_{j+1}$, \dots, $x_k$ ]
\item[$\bullet$]
append $x'$ to \texttt{monomials}$[p,m]$
\end{itemize}
\end{itemize}
\end{enumerate}
\item
eliminate any repetitions from \texttt{monomials}$[p,m]$
\end{enumerate}
\end{enumerate}
\end{itemize}
\vspace{-10pt}
\caption{Algorithm for generating $n$-ary one-operator monomials}
\label{operatormonomials}
\end{figure}


\subsubsection*{Consequences of the Rota-Baxter relation}

See Example \ref{exampledendriformconsequences}.
The consequences of the RB relation in $\au_{3,1}(5;4)$ form a spanning set 
for the homogeneous subspace $\mathcal{I}_{3,1}(5,4)$; 
see Definition \ref{RBidealdefinition}.
(Only 144 of the 378 basis monomials for $\au_{3,1}(5;4)$ appear in the consequences of the RB relation.)
We construct the matrix $C$ of consequences of the RB relation which has size $42 \times 378$;
each row of $C$ is the coefficient vector of one consequence of the RB relation with respect to
the ordered basis of one-operator monomials.


\subsubsection*{The rewriting morphism}

The rewriting morphism $r_{3,1} \colon \na_{3,3} \to \au_{3,1}$ is defined on the generating operations
by Example \ref{rewritingternarydendriform},
and extends uniquely since $\na_{3,3}$ is a free operad.
We consider the restriction $r'_{3,1}$ of the rewriting morphism to the quadratic subspace of its domain.
Clearly $r'_{3,1}( \na_{3,3}(5) ) \subseteq \au_{3,1}(5;4)$.
Table \ref{dendriformrewriting} displays the 27 basis monomials in $\na_{3,3}(5)$
in lex order together with their images $r'_{3,1}(x_{ijk})$
under the rewriting morphism and the Dyck words corresponding to the operator monomials.
(There are only 23 distinct operator monomials in Table \ref{dendriformrewriting};
the rewriting morphism is not injective.)
We construct the rewriting matrix $X$ which has size $27 \times 378$.
Each row of $X$ corresponds to the rewriting of one basis monomial in $\na_{3,3}(5)$; 
thus each row of $X$ is zero except for a single 1.

The composition $s_{3,1} \circ r'_{3,1}$
(Definition \ref{compositiondefinition}) 
is a linear map from
$\na_{3,3}(5)$ to $\rb_{3,1}(5,4)$
(Definition \ref{RBidealdefinition}).
The kernel of this composition is the space of quadratic relations satisfied by
(the rewritings of) the generating operations of $\na$.
A vector $v$ in the domain of this composition is a relation
satisfied by the generating operations precisely when $vX$ belongs to the row space of $C$.

\begin{table}[ht]
$
\begin{array}{lll}
x_{ijk} \in \na_{3,3}(5)  &\qquad  r_{3,1}( x_{ijk} ) \in \au_{3,1}(5;4)  &\qquad  \text{Dyck word} \\
\midrule
\omega_1 \circ_1 \omega_1  &\qquad  {\ast}S({\ast})S({\ast})S({\ast})S({\ast})  &\qquad  ()(())(())(())(())  \\ 
\omega_2 \circ_1 \omega_1  &\qquad  S({\ast}S({\ast})S({\ast})){\ast}S({\ast})  &\qquad  (()(())(()))()(())  \\ 
\omega_3 \circ_1 \omega_1  &\qquad  S({\ast}S({\ast})S({\ast}))S({\ast}){\ast}  &\qquad  (()(())(()))(())()  \\ 
\omega_1 \circ_1 \omega_2  &\qquad  S({\ast}){\ast}S({\ast})S({\ast})S({\ast})  &\qquad  (())()(())(())(())  \\ 
\omega_2 \circ_1 \omega_2  &\qquad  S(S({\ast}){\ast}S({\ast})){\ast}S({\ast})  &\qquad  ((())()(()))()(())  \\ 
\omega_3 \circ_1 \omega_2  &\qquad  S(S({\ast}){\ast}S({\ast}))S({\ast}){\ast}  &\qquad  ((())()(()))(())()  \\ 
\omega_1 \circ_1 \omega_3  &\qquad  S({\ast})S({\ast}){\ast}S({\ast})S({\ast})  &\qquad  (())(())()(())(())  \\ 
\omega_2 \circ_1 \omega_3  &\qquad  S(S({\ast})S({\ast}){\ast}){\ast}S({\ast})  &\qquad  ((())(())())()(())  \\ 
\omega_3 \circ_1 \omega_3  &\qquad  S(S({\ast})S({\ast}){\ast})S({\ast}){\ast}  &\qquad  ((())(())())(())()  \\ 
\omega_1 \circ_2 \omega_1  &\qquad  {\ast}S({\ast}S({\ast})S({\ast}))S({\ast})  &\qquad  ()(()(())(()))(())  \\ 
\omega_2 \circ_2 \omega_1  &\qquad  S({\ast}){\ast}S({\ast})S({\ast})S({\ast})  &\qquad  (())()(())(())(())  \\ 
\omega_3 \circ_2 \omega_1  &\qquad  S({\ast})S({\ast}S({\ast})S({\ast})){\ast}  &\qquad  (())(()(())(()))()  \\ 
\omega_1 \circ_2 \omega_2  &\qquad  {\ast}S(S({\ast}){\ast}S({\ast}))S({\ast})  &\qquad  ()((())()(()))(())  \\ 
\omega_2 \circ_2 \omega_2  &\qquad  S({\ast})S({\ast}){\ast}S({\ast})S({\ast})  &\qquad  (())(())()(())(())  \\ 
\omega_3 \circ_2 \omega_2  &\qquad  S({\ast})S(S({\ast}){\ast}S({\ast})){\ast}  &\qquad  (())((())()(()))()  \\ 
\omega_1 \circ_2 \omega_3  &\qquad  {\ast}S(S({\ast})S({\ast}){\ast})S({\ast})  &\qquad  ()((())(())())(())  \\ 
\omega_2 \circ_2 \omega_3  &\qquad  S({\ast})S({\ast})S({\ast}){\ast}S({\ast})  &\qquad  (())(())(())()(())  \\ 
\omega_3 \circ_2 \omega_3  &\qquad  S({\ast})S(S({\ast})S({\ast}){\ast}){\ast}  &\qquad  (())((())(())())()  \\ 
\omega_1 \circ_3 \omega_1  &\qquad  {\ast}S({\ast})S({\ast}S({\ast})S({\ast}))  &\qquad  ()(())(()(())(()))  \\ 
\omega_2 \circ_3 \omega_1  &\qquad  S({\ast}){\ast}S({\ast}S({\ast})S({\ast}))  &\qquad  (())()(()(())(()))  \\ 
\omega_3 \circ_3 \omega_1  &\qquad  S({\ast})S({\ast}){\ast}S({\ast})S({\ast})  &\qquad  (())(())()(())(())  \\ 
\omega_1 \circ_3 \omega_2  &\qquad  {\ast}S({\ast})S(S({\ast}){\ast}S({\ast}))  &\qquad  ()(())((())()(()))  \\ 
\omega_2 \circ_3 \omega_2  &\qquad  S({\ast}){\ast}S(S({\ast}){\ast}S({\ast}))  &\qquad  (())()((())()(()))  \\ 
\omega_3 \circ_3 \omega_2  &\qquad  S({\ast})S({\ast})S({\ast}){\ast}S({\ast})  &\qquad  (())(())(())()(())  \\ 
\omega_1 \circ_3 \omega_3  &\qquad  {\ast}S({\ast})S(S({\ast})S({\ast}){\ast})  &\qquad  ()(())((())(())())  \\ 
\omega_2 \circ_3 \omega_3  &\qquad  S({\ast}){\ast}S(S({\ast})S({\ast}){\ast})  &\qquad  (())()((())(())())  \\ 
\omega_3 \circ_3 \omega_3  &\qquad  S({\ast})S({\ast})S({\ast})S({\ast}){\ast}  &\qquad  (())(())(())(())()  
\end{array}
$
\bigskip
\caption{The rewriting morphism for the ternary dendriform operad}
\label{dendriformrewriting}
\end{table}

\begin{theorem}
\label{ternarydendriformrelations}
Every quadratic relation satisfied by the defining operations for the ternary analog
of dendriform algebras is a linear combination of the following 10 linearly independent relations:
\begin{alignat*}{2}
(({\ast}{\ast}{\ast})_2{\ast}{\ast})_1 &= ({\ast}({\ast}{\ast}{\ast})_1{\ast})_2, 
&\qquad
(({\ast}{\ast}{\ast})_3{\ast}{\ast})_1 &= ({\ast}{\ast}({\ast}{\ast}{\ast})_1)_3, 
\\
({\ast}({\ast}{\ast}{\ast})_3{\ast})_2 &= ({\ast}{\ast}({\ast}{\ast}{\ast})_2)_3, 
&\qquad
({\ast}({\ast}{\ast}{\ast})_2{\ast})_2 &= ({\ast}{\ast}({\ast}{\ast}{\ast})_1)_3, 
\\
(({\ast}{\ast}{\ast})_1{\ast}{\ast})_1 &= ({\ast}({\ast}{\ast}{\ast})_\bullet{\ast})_1, 
&\qquad
(({\ast}{\ast}{\ast})_1{\ast}{\ast})_1 &= ({\ast}{\ast}({\ast}{\ast}{\ast})_\bullet)_1, 
\\
({\ast}{\ast}({\ast}{\ast}{\ast})_3)_3 &= ({\ast}({\ast}{\ast}{\ast})_\bullet{\ast})_3,
&\qquad
({\ast}{\ast}({\ast}{\ast}{\ast})_3)_3 &= (({\ast}{\ast}{\ast})_\bullet{\ast}{\ast})_3,
\\
(({\ast}{\ast}{\ast})_2{\ast}{\ast})_1 &= ({\ast}{\ast}({\ast}{\ast}{\ast})_\bullet)_2, 
&\qquad
({\ast}{\ast}({\ast}{\ast}{\ast})_2)_3 &= (({\ast}{\ast}{\ast})_\bullet{\ast}{\ast})_2. 
\end{alignat*}
We have introduced the auxiliary sum operation
$({\ast},{\ast},{\ast})_\bullet := \sum_{i=1}^3 ({\ast}{\ast}{\ast})_i$
which is associative by Proposition \ref{sumoperationassociative}.
\end{theorem}

\begin{proof}
These computations were done using rational arithmetic exclusively.
We construct the block matrix 
\[
B = \begin{bmatrix} C & O_{42,27} \\ X & I_{27} \end{bmatrix},
\]
where $O$ and $I$ are the zero and identity matrices of the indicated sizes.
We compute the RCF of this matrix (also called the reduced row-echelon form);
the matrix has full rank.
This is the explicit matrix form of computing the composition of the rewriting morphism
with the natural surjection onto the Rota-Baxter operad.
We identify the rows whose leading 1s appear in the right part of the matrix (columns $j > 378$);
these rows represent a basis for the kernel of the composition.
There are 10 such rows (the bottom 10 rows of the RCF).
We extract the $10 \times 27$ submatrix containing the right parts of these rows;
see Figure \ref{dendriformmatrix} (with dot for zero).
The rows are the coefficient vectors of the quadratic relations satisfied by
the operations defining the ternary analog of dendriform algebras;
the nonzero coefficients are $\pm 1$.
The squared Euclidean norms of these vectors (sorted by length) are
$2^3$, $4^6$, 6 (the exponents indicate multiplicities).
Applying the LLL algorithm for lattice basis reduction to these rows,
we obtain a shorter basis whose squared norms are $2^4$, $4^6$
(we have eliminated the relation with 6 terms).
These reduced relations are stated in the theorem.
(For the application of the LLL algorithm to polynomial identities see Bremner and Peresi \cite{BP}.)
\end{proof}

\begin{figure}[ht]
\small
$\left[
\begin{array}{rrrrrrrrrrrrrrrrrrrrrrrrrrr}
1 &\!\!\! . &\!\!\! . &\!\!\! . &\!\!\! . &\!\!\! . &\!\!\! . &\!\!\! . &\!\!\! . &\!\!\! . &\!\!\! . &\!\!\! . &\!\!\! . &\!\!\! . &\!\!\! . &\!\!\! . &\!\!\! . &\!\!\! . &\!\!\! -1 &\!\!\! . &\!\!\! . &\!\!\! -1 &\!\!\! . &\!\!\! . &\!\!\! -1 &\!\!\! . &\!\!\! . \\
. &\!\!\! 1 &\!\!\! . &\!\!\! . &\!\!\! 1 &\!\!\! . &\!\!\! . &\!\!\! 1 &\!\!\! . &\!\!\! . &\!\!\! . &\!\!\! . &\!\!\! . &\!\!\! . &\!\!\! . &\!\!\! . &\!\!\! . &\!\!\! . &\!\!\! . &\!\!\! . &\!\!\! . &\!\!\! . &\!\!\! . &\!\!\! -1 &\!\!\! . &\!\!\! . &\!\!\! . \\
. &\!\!\! . &\!\!\! 1 &\!\!\! . &\!\!\! . &\!\!\! 1 &\!\!\! . &\!\!\! . &\!\!\! 1 &\!\!\! . &\!\!\! . &\!\!\! . &\!\!\! . &\!\!\! . &\!\!\! . &\!\!\! . &\!\!\! . &\!\!\! . &\!\!\! . &\!\!\! . &\!\!\! . &\!\!\! . &\!\!\! . &\!\!\! . &\!\!\! . &\!\!\! . &\!\!\! -1 \\
. &\!\!\! . &\!\!\! . &\!\!\! 1 &\!\!\! . &\!\!\! . &\!\!\! . &\!\!\! . &\!\!\! . &\!\!\! . &\!\!\! . &\!\!\! . &\!\!\! . &\!\!\! . &\!\!\! . &\!\!\! . &\!\!\! . &\!\!\! . &\!\!\! . &\!\!\! -1 &\!\!\! . &\!\!\! . &\!\!\! -1 &\!\!\! . &\!\!\! . &\!\!\! -1 &\!\!\! . \\
. &\!\!\! . &\!\!\! . &\!\!\! . &\!\!\! . &\!\!\! . &\!\!\! 1 &\!\!\! . &\!\!\! . &\!\!\! . &\!\!\! . &\!\!\! . &\!\!\! . &\!\!\! . &\!\!\! . &\!\!\! . &\!\!\! . &\!\!\! . &\!\!\! . &\!\!\! . &\!\!\! -1 &\!\!\! . &\!\!\! . &\!\!\! . &\!\!\! . &\!\!\! . &\!\!\! . \\
. &\!\!\! . &\!\!\! . &\!\!\! . &\!\!\! . &\!\!\! . &\!\!\! . &\!\!\! . &\!\!\! . &\!\!\! 1 &\!\!\! . &\!\!\! . &\!\!\! 1 &\!\!\! . &\!\!\! . &\!\!\! 1 &\!\!\! . &\!\!\! . &\!\!\! -1 &\!\!\! . &\!\!\! . &\!\!\! -1 &\!\!\! . &\!\!\! . &\!\!\! -1 &\!\!\! . &\!\!\! . \\
. &\!\!\! . &\!\!\! . &\!\!\! . &\!\!\! . &\!\!\! . &\!\!\! . &\!\!\! . &\!\!\! . &\!\!\! . &\!\!\! 1 &\!\!\! . &\!\!\! . &\!\!\! . &\!\!\! . &\!\!\! . &\!\!\! . &\!\!\! . &\!\!\! . &\!\!\! -1 &\!\!\! . &\!\!\! . &\!\!\! -1 &\!\!\! . &\!\!\! . &\!\!\! -1 &\!\!\! . \\
. &\!\!\! . &\!\!\! . &\!\!\! . &\!\!\! . &\!\!\! . &\!\!\! . &\!\!\! . &\!\!\! . &\!\!\! . &\!\!\! . &\!\!\! 1 &\!\!\! . &\!\!\! . &\!\!\! 1 &\!\!\! . &\!\!\! . &\!\!\! 1 &\!\!\! . &\!\!\! . &\!\!\! . &\!\!\! . &\!\!\! . &\!\!\! . &\!\!\! . &\!\!\! . &\!\!\! -1 \\
. &\!\!\! . &\!\!\! . &\!\!\! . &\!\!\! . &\!\!\! . &\!\!\! . &\!\!\! . &\!\!\! . &\!\!\! . &\!\!\! . &\!\!\! . &\!\!\! . &\!\!\! 1 &\!\!\! . &\!\!\! . &\!\!\! . &\!\!\! . &\!\!\! . &\!\!\! . &\!\!\! -1 &\!\!\! . &\!\!\! . &\!\!\! . &\!\!\! . &\!\!\! . &\!\!\! . \\
. &\!\!\! . &\!\!\! . &\!\!\! . &\!\!\! . &\!\!\! . &\!\!\! . &\!\!\! . &\!\!\! . &\!\!\! . &\!\!\! . &\!\!\! . &\!\!\! . &\!\!\! . &\!\!\! . &\!\!\! . &\!\!\! 1 &\!\!\! . &\!\!\! . &\!\!\! . &\!\!\! . &\!\!\! . &\!\!\! . &\!\!\! -1 &\!\!\! . &\!\!\! . &\!\!\! .
\end{array}
\right]$
\caption{The coefficient vectors of the ternary dendriform relations}
\label{dendriformmatrix}
\end{figure}

\begin{definition}
The nonsymmetric operad $\tad$ defining the ternary analog of dendriform algebras
is generated by three nonassociative ternary operations $({\ast}{\ast}{\ast})_i$ ($1 \le i \le 3$)
satisfying the relations of 
Theorem \ref{ternarydendriformrelations}.
\end{definition}

We remark that $\tad$ may be regarded as the quotient of $\na_{3,3}$ by the ideal generated
by the 10 relations of Theorem \ref{ternarydendriformrelations}.
The composition $f_{3,1} = s_{3,1} \circ r_{3,1}$ (Definition \ref{compositiondefinition}) 
defines a surjective operad morphism $g_{3,1}$
from $\tad$ to the suboperad of $\rb_{3,1}$ generated by the rewritings of the generating 
operations (Example \ref{rewritingternarydendriform}).
Theorem \ref{ternarydendriformrelations}
identifies the full quadratic kernel of $f_{3,1}$. 
Hence the induced morphism $g_{3,1}$ has trivial quadratic kernel.

\begin{openproblem}
Is the morphism $g_{3,1}$ injective?
That is, are there any higher (cubic or above) relations satisfied by the rewritings
of the generating operations that are not consequences of the quadratic relations?
\end{openproblem}

\begin{openproblem}
Find a Gr\"obner basis for the operad $\tad$.
Does this operad have a quadratic Gr\"obner basis?
\end{openproblem}

\begin{openproblem}
Is this operad Koszul?
(This would follow from the existence of a quadratic Gr\"obner basis;
see Hoffbeck \cite{Hoffbeck}.)
\end{openproblem}


\subsubsection*{Dimension calculations for the operad $\tad$}

We first recall the dimension formula for the dendriform operad $\mathcal{D}$:
\[
\dim \mathcal{D}(n) = \frac1n \binom{2n-2}{n-1}.
\]
These are the shifted Catalan numbers; see \cite[\S 3]{LR}.

For the ternary analog $\tad$ we clearly have
\[
\dim \tad(1) = 1, \qquad
\dim \tad(3) = 3, \qquad
\dim \tad(5) = 17 \, ( = 27 - 10 ).
\]
The last dimension is $\dim \na_{3,3}(5) - \dim \ker f_{3,1}$.
We next describe computations which demonstrate the following results:
\[
\dim \tad(7) = 119, \qquad
\dim \tad(9) = 929.
\]
These computations were done using rational arithmetic exclusively.

In arity 7 (weight 3) there are 12 cubic association types in lex order using 
the compositions of 7 into 3 odd parts (511, 331, 313, 151, 133, 115):
\begin{alignat*}{4}
&
( ( ( {\ast} {\ast} {\ast} ) {\ast} {\ast} ) {\ast} {\ast} )
&\qquad
&
( ( {\ast} ( {\ast} {\ast} {\ast} ) {\ast} ) {\ast} {\ast} )
&\qquad
&
( ( {\ast} {\ast} ( {\ast} {\ast} {\ast} ) ) {\ast} {\ast} )
&\qquad
&
( ( {\ast} {\ast} {\ast} ) ( {\ast} {\ast} {\ast} ) {\ast} )
\\
&
( ( {\ast} {\ast} {\ast} ) {\ast} ( {\ast} {\ast} {\ast} ) )
&\qquad
&
( {\ast} ( ( {\ast} {\ast} {\ast} ) {\ast} {\ast} ) {\ast} )
&\qquad
&
( {\ast} ( {\ast} ( {\ast} {\ast} {\ast} ) {\ast} ) {\ast} )
&\qquad
&
( {\ast} ( {\ast} {\ast} ( {\ast} {\ast} {\ast} ) ) {\ast} )
\\
&
( {\ast} ( {\ast} {\ast} {\ast} ) ( {\ast} {\ast} {\ast} ) )
&\qquad
&
( {\ast} {\ast} ( ( {\ast} {\ast} {\ast} ) {\ast} {\ast} ) )
&\qquad
&
( {\ast} {\ast} ( {\ast} ( {\ast} {\ast} {\ast} ) {\ast} ) )
&\qquad
&
( {\ast} {\ast} ( {\ast} {\ast} ( {\ast} {\ast} {\ast} ) ) )
\end{alignat*}
Altogether we have $12 \cdot 3^3 = 324$ cubic nonsymmetric monomials in arity 7. 
Let $\rho$ denote one of the 10 quadratic relations in Theorem \ref{ternarydendriformrelations},
and let $\omega_i$ ($1 \le i \le 3$) denote the three generating operations of $\tad$.
There are 15 partial compositions of the form $\rho \, \circ_k \, \omega_i$ 
where $1 \le k \le 5$, 
and 9 partial compositions of the form $\omega_i \, \circ_k \, \rho$ where $1 \le k \le 3$.
Altogether, there are 24 cubic consequences (arity 7) of each relation $\rho$, 
and hence 240 cubic consequences of all 10 relations;
we will denote these cubic consequences by $\sigma_i$, ordered arbitrarily.
We construct a matrix of size $240 \times 324$. 
The $(i,j)$ entry is the coefficient of the $j$-th cubic monomial in $\sigma_i$.
The rank of this matrix is 205 and hence $\dim \tad(7) = 324 - 205$.

In arity 9 (weight 4) there are 55 quartic association types which can be generated
recursively in lex order using the 10 compositions of 9 into 3 odd parts. 
Hence we have $55 \cdot 3^4 = 4455$ quartic nonsymmetric monomials.
There are $240 \cdot 7 \cdot 3 = 5040$ partial compositions
$\sigma_i \, \circ_k \, \omega_j$, 
and $3 \cdot 3 \cdot 240 = 2160$ partial compositions $\omega_j \, \circ_k \, \sigma_i$.
Altogether 7200 consequences in arity 9,
but only 4950 of these are distinct. 
We denote these consequences by $\tau_i$, ordered arbitrarily.
We construct a matrix of size $4950 \times 4455$. 
The $(i,j)$ entry is the coefficient of the $j$-th quartic monomial in $\tau_i$.
The rank of this matrix is 3526 and hence $\dim \tad(9) = 4455 - 3526 = 929$.

\begin{conjecture}
We have
\[
\dim \tad(2n-1) = \frac1n \sum_{i=1}^{n-1} \binom{2n-2}{i} \binom{3n-2-i}{2n-1}.
\]
This is OEIS sequence A344553, which begins 
1, 3, 17, 119, 929, 7755, 67745, 611567, \dots.
According to the OEIS, this sequence enumerates the lattice paths from $(0,0)$ to $(2n-1,n)$ 
using steps $E=(1,0)$, $N=(0,1)$, and $D=(1,1)$ which stay weakly above the line through 
$(0,0)$ and $(2n-1,n)$.
See von Bell and Yip \cite{vBY}.
\end{conjecture}


\section{The ternary analog of quadri-algebras}
\label{ternaryquadrialgebras}


\subsubsection*{The restriction of the rewriting morphism}

The operad $\na_{3,9}$ is generated by nine nonassociative ternary operations 
which we write as $\omega_{ij} \in \na_{3,9}(3)$ for $1 \le i, j \le 3$.
We display these operations together with their images under the restriction of the rewriting morphism
$r_{3,2} \colon \na_{3,9}(3) \to \au_{3,2}(3;2,2)$
in the following $3 \times 3$ array 
(this was previously given as Example \ref{rewritingquadrialgebras}).
For ease of reading we use parentheses for $S$ and square brackets for $T$:
\[
\begin{array}{lll}
\omega_{11} \,\mapsto\, {\ast} T[S({\ast})] T[S({\ast})], \quad 
\omega_{12} \,\mapsto\, T[{\ast}] S({\ast}) T[S({\ast})], \quad 
\omega_{13} \,\mapsto\, T[{\ast}] T[S({\ast})] S({\ast}), 
\\[3pt]
\omega_{21} \,\mapsto\, S({\ast}) T[{\ast}] T[S({\ast})], \quad 
\omega_{22} \,\mapsto\, T[S({\ast})] {\ast} T[S({\ast})], \quad 
\omega_{23} \,\mapsto\, T[S({\ast})] T[{\ast}] S({\ast}), 
\\[3pt]
\omega_{31} \,\mapsto\, S({\ast}) T[S({\ast})] T[{\ast}], \quad 
\omega_{32} \,\mapsto\, T[S({\ast})] S({\ast}) T[{\ast}], \quad
\omega_{33} \,\mapsto\, T[S({\ast})] T[S({\ast})] {\ast}.
\end{array}
\]
Since $\na_{n,k}$ is a free operad, this array suffices to define the rewriting morphism
$r_{3,2} \colon \na_{3,9} \to \au_{3,2}$.
In $\omega_{ij}$ the subscript $ij$ means that
$S$ occurs on every argument except $i$, and
$T$ occurs on every argument except $j$.
We will usually write $( {\ast} {\ast} {\ast} )_{ij}$ instead of
$\omega_{ij}( {\ast} {\ast} {\ast} )$.
Transposing the preceding array corresponds to interchanging $S$ and $T$
(and using the commutativity assumption $ST = TS$).

For later convenience, we define seven auxiliary operations:
the row sums, the column sums, and the total sum:
\[
\omega_{i \bullet} := \sum_{j=1}^3 \omega_{ij} \; (i = 1,2,3),
\quad
\omega_{\bullet j} := \sum_{i=1}^3 \omega_{ij} \; (j = 1,2,3),
\quad
\omega_{\bullet\bullet} := \sum_{i=1}^3 \sum_{j=1}^3 \omega_{ij}.
\]
We will show later that the total sum operation is associative (Proposition \ref{totalassociative}) 
and that the three row (column) sums satisfy the defining relations for the ternary analog of dendriform algebras
(Proposition \ref{rowcoldendriform}).

As basis of the quadratic subspace $\na_{3,9}(5)$ we use the 243 monomials 
$\omega_{ij} \circ_k \omega_{\ell m}$
for $1 \le i, j, k, \ell, m \le 3$;
we apply lex order on the quintuples $(k,\ell,m,i,j)$.
The homogeneous subspace $\na_{3,9}(5)$ consists of all formal quadratic expressions 
in the generating operations $\omega_{ij}$.
The kernel of the composition of the restriction of the rewriting morphism with the surjection 
onto the Rota-Baxter ideal (Definition \ref{compositiondefinition})
is the space of quadratic relations satisfied by the rewritings of the nine operations.
We present some examples of rewriting of quadratic monomials:
\begin{align*}
( ( {\ast} {\ast} {\ast} )_{13} {\ast} {\ast} )_{23} &\;\longmapsto\; 
T[S(T[{\ast}]T[S({\ast})]S({\ast}))]T[{\ast}]S({\ast}),  
\\
( ( {\ast} {\ast} {\ast} )_{23} {\ast} {\ast} )_{13} &\;\longmapsto\; 
T[T[S({\ast})]T[{\ast}]S({\ast})]T[S({\ast})]S({\ast}), 
\\
( ( {\ast} {\ast} {\ast} )_{32} {\ast} {\ast} )_{33} &\;\longmapsto\; 
T[S(T[S({\ast})]S({\ast})T[{\ast}])]T[S({\ast})]{\ast},  
\\
( {\ast} ( {\ast} {\ast} {\ast} )_{12} {\ast} )_{23} &\;\longmapsto\; 
T[S({\ast})]T[T[{\ast}]S({\ast})T[S({\ast})]]S({\ast}),  
\\
( {\ast} ( {\ast} {\ast} {\ast} )_{22} {\ast} )_{13} &\;\longmapsto\; 
T[{\ast}]T[S(T[S({\ast})]{\ast}T[S({\ast})])]S({\ast}),  
\\
( {\ast} ( {\ast} {\ast} {\ast} )_{31} {\ast} )_{33} &\;\longmapsto\; 
T[S({\ast})]T[S(S({\ast})T[S({\ast})]T[{\ast}])]{\ast},  
\\
( {\ast} {\ast} ( {\ast} {\ast} {\ast} )_{11} )_{23} &\;\longmapsto\; 
T[S({\ast})]T[{\ast}]S({\ast}T[S({\ast})]T[S({\ast})]),  
\\
( {\ast} {\ast} ( {\ast} {\ast} {\ast} )_{21} )_{13} &\;\longmapsto\; 
T[{\ast}]T[S({\ast})]S(S({\ast})T[{\ast}]T[S({\ast})]),  
\\
( {\ast} {\ast} ( {\ast} {\ast} {\ast} )_{23} )_{33} &\;\longmapsto\; 
T[S({\ast})]T[S({\ast})]T[S({\ast})]T[{\ast}]S({\ast}),  
\\
( {\ast} {\ast} ( {\ast} {\ast} {\ast} )_{33} )_{23} &\;\longmapsto\; 
T[S({\ast})]T[{\ast}]S(T[S({\ast})]T[S({\ast})]{\ast}).
\end{align*}
The rewriting morphism is not injective: the 243 quadratic basis monomials produce
only 235 distinct rewritings.


\subsubsection*{Codomain of the restriction of the rewriting morphism}

The operad $\au_{3,2}$ is generated by an associative ternary operation $A$ and 
two commuting unary operators $S$ and $T$.
The homogeneous subspace $\au_{3,2}(5;4,4)$ is the codomain of the restriction 
of the rewriting morphism to the quadratic subspace $\na_{3,9}(5)$.
For the domain $\na_{3,9}(5)$ the rewritings of the basis monomials have arity 5 
and operator multiplicity 4 for both $S$ and $T$.
A basis of $\au_{3,2}(5;4,4)$ consists of the 100548 operator monomials of arity 5 and 
operator multiplicities $\mathbf{m} = (4,4)$;
see Examples \ref{dimauexample1} and \ref{dimauexample2}.
We do not compute all these monomials;
instead, we compute the consequences of the RB relations for $S$ and $T$ and use the much smaller set
(with 53520 elements) of two-operator monomials which occur in the terms of these consequences.
We also include the monomials which occur in the image of the rewriting morphism.
We order the union of these two sets arbitrarily.


\subsubsection*{Consequences of the Rota-Baxter relation}

The consequences in $\au_{3,2}(5;4,4)$ of the RB relations for $S$ and $T$ form a spanning set 
for the homogeneous subspace $\mathcal{I}_{3,2}(5;4,4)$; 
see Definition \ref{RBidealdefinition}.
For computational convenience we represent the two-operator monomials using strings.
If $x$ is a two-operator monomial of arity $n$, regarded as a string over the ordered alphabet
$\{ \ast, S, T, (, ), [, ] \}$ (where $\ast$ is the generic argument symbol), 
the possible partial compositions with $A, S, T$ are defined as follows: 
\begin{alignat*}{2}
&
A \circ_i x \; (i = 1,2,3)
&\qquad
&
\text{replace $x$ by 
$x{\ast}{\ast} \, (i = 1)$,
${\ast}x{\ast} \, (i = 2)$,
${\ast}{\ast}x \, (i = 3)$}, 
\\
&
x \circ_i A \; (i=1,\dots,n)
&\qquad
&
\text{replace the $i$-th occurrence of $\ast$ in $x$ by ${\ast}{\ast}{\ast}$}
\\
&
S \circ_1 x
&\qquad
&
\text{replace $x$ by $S(x)$}
\\
&
x \circ_i S \; (i=1,\dots,n)
&\qquad
&
\text{replace the $i$-th occurrence of $\ast$ in $x$ by $S({\ast})$}
\\
&
T \circ_1 x
&\qquad
&
\text{replace $x$ by $T[x]$}
\\
&
x \circ_i T \; (i=1,\dots,n)
&\qquad
&
\text{replace the $i$-th occurrence of $\ast$ in $x$ by $T[{\ast}]$}
\end{alignat*}
In these strings, the symbols $S$ and $T$ are redundant; we may replace $S($ by $($
and $T[$ by $[$.
The algorithm in Figure \ref{generatingconsequences} 
generates the consequences of the RB relations.

The algorithm for computing consequences of the RB relations is highly redundant
(owing to the parallel and sequential properties of partial compositions
and the associativity of the ternary operation).
For each operator $S$ and $T$ we obtain 638400 consequences, 
summing over all permutations of the corresponding multiset (see Example \ref{multisetexample}), 
but in each case we obtain only 13860 distinct consequences.
Combining the consequences for $S$ and $T$ produces 17640
distinct consequences, which we order arbitrarily.
These consequences contain only 53520 distinct monomials
(in normal form with respect to the commutativity of $S$ and $T$, see below),
and this allows us to decrease the size of the matrix computations.

\begin{example}
Recall the ternary RB relation $R$ for the unary operator $S$:
\[
R :=
S({\ast})S({\ast})S({\ast}) 
- S({\ast}S({\ast})S({\ast})) 
- S(S({\ast}){\ast}S({\ast})) 
- S(S({\ast})S({\ast}){\ast}).
\]
The following are the partial compositions of $R$ with $A$ and $T$:
\begin{align*}
A \circ_1 R &= 
S({\ast})S({\ast})S({\ast}){\ast}{\ast} 
- S({\ast}S({\ast})S({\ast})){\ast}{\ast} 
- S(S({\ast}){\ast}S({\ast})){\ast}{\ast} 
- S(S({\ast})S({\ast}){\ast}){\ast}{\ast},
\\
A \circ_2 R &= 
{\ast}S({\ast})S({\ast})S({\ast}){\ast} 
- {\ast}S({\ast}S({\ast})S({\ast})){\ast}
- {\ast}S(S({\ast}){\ast}S({\ast})){\ast} 
- {\ast}S(S({\ast})S({\ast}){\ast}){\ast},
\\
A \circ_3 R &= 
{\ast}{\ast}S({\ast})S({\ast})S({\ast}) 
- {\ast}{\ast}S({\ast}S({\ast})S({\ast}))
- {\ast}{\ast}S(S({\ast}){\ast}S({\ast}))
- {\ast}{\ast}S(S({\ast})S({\ast}){\ast}),
\\
R \circ_1 A &= 
S({\ast}{\ast}{\ast})S({\ast})S({\ast}) 
- S({\ast}{\ast}{\ast}S({\ast})S({\ast})) 
- S(S({\ast}{\ast}{\ast}){\ast}S({\ast})) 
- S(S({\ast}{\ast}{\ast})S({\ast}){\ast}),
\\
R \circ_2 A &= 
S({\ast})S({\ast}{\ast}{\ast})S({\ast}) 
- S({\ast}S({\ast}{\ast}{\ast})S({\ast})) 
- S(S({\ast}){\ast}{\ast}{\ast}S({\ast})) 
- S(S({\ast})S({\ast}{\ast}{\ast}){\ast}),
\\
R \circ_3 A &= 
S({\ast})S({\ast})S({\ast}{\ast}{\ast}) 
- S({\ast}S({\ast})S({\ast}{\ast}{\ast})) 
- S(S({\ast}){\ast}S({\ast}{\ast}{\ast})) 
- S(S({\ast})S({\ast}){\ast}{\ast}{\ast}),
\\\
T \circ_1 R &=  
T[S({\ast})S({\ast})S({\ast})]
- T[S({\ast}S({\ast})S({\ast}))]
- T[S(S({\ast}){\ast}S({\ast}))]
\\
&\quad
- T[S(S({\ast})S({\ast}){\ast})],
\\
R \circ_1 T &= 
S(T[{\ast}])S({\ast})S({\ast}) 
- S(T[{\ast}]S({\ast})S({\ast})) 
- S(S(T[{\ast}]){\ast}S({\ast})) 
\\
&\quad
- S(S(T[{\ast}])S({\ast}){\ast}),
\\
R \circ_2 T &= 
S({\ast})S(T[{\ast}])S({\ast}) 
- S({\ast}S(T[{\ast}])S({\ast})) 
- S(S({\ast})T[{\ast}]S({\ast})) 
\\
&\quad
- S(S({\ast})S(T[{\ast}]){\ast}),
\\
R \circ_3 T &= 
S({\ast})S({\ast})S(T[{\ast}]) 
- S({\ast}S({\ast})S(T[{\ast}])) 
- S(S({\ast}){\ast}S(T[{\ast}])) 
\\
&\quad
- S(S({\ast})S({\ast})T[{\ast}]).
\end{align*}
\end{example}

\begin{example}
\label{excon}
The RB relations $R_S$ and $R_T$ for $S$ and $T$ have  
tridegrees $(3,3,0)$ and $(3,0,3)$ respectively.
In the computations that follow we will be concerned only with the consequences of
$R_S$ and $R_T$ in tridegree $(5,4,4)$.
To produce the consequences of $R_S$ we need to apply partial compositions
corresponding to some permutation of the multiset $(A,S,T,T,T,T)$;
for $R_T$ we use $(A,S,S,S,S,T)$.
For example, consider the following consequence of $R_S$:
\[
( ( S \circ_1 ( ( ( R_S \circ_1 T ) \circ_2 T ) \circ_3 A ) ) \circ_5 T ) \circ_3 T.
\]
We start with $R_0 = R_S$ and illustrate each step in the computation:
\begin{align*}
R_0 &=
S({\ast})S({\ast})S({\ast}) 
- S({\ast}S({\ast})S({\ast})) 
- S(S({\ast}){\ast}S({\ast})) 
- S(S({\ast})S({\ast}){\ast}),
\\
R_1 = R_0 \circ_1 T &=
S(T[{\ast}])S({\ast})S({\ast}) 
- S(T[{\ast}]S({\ast})S({\ast})) 
- S(S(T[{\ast}]){\ast}S({\ast})) 
\\
&\quad {}
- S(S(T[{\ast}])S({\ast}){\ast}),
\\
R_2 = R_1 \circ_2 T &=
S(T[{\ast}])S(T[{\ast}])S({\ast}) 
- S(T[{\ast}]S(T[{\ast}])S({\ast})) 
- S(S(T[{\ast}])T[{\ast}]S({\ast}))
\\
&\quad {} 
- S(S(T[{\ast}])S(T[{\ast}]){\ast}),
\\
R_3 = R_2 \circ_3 A &=
S(T[{\ast}])S(T[{\ast}])S({\ast}{\ast}{\ast}) 
- S(T[{\ast}]S(T[{\ast}])S({\ast}{\ast}{\ast})) 
\\
&\quad {}
- S(S(T[{\ast}])T[{\ast}]S({\ast}{\ast}{\ast})) 
- S(S(T[{\ast}])S(T[{\ast}]){\ast}{\ast}{\ast}),
\\
R_4 = S \circ_1 R_3 &=
S(S(T[{\ast}])S(T[{\ast}])S({\ast}{\ast}{\ast}))
- S(S(T[{\ast}]S(T[{\ast}])S({\ast}{\ast}{\ast}))) 
\\
&\quad {}
- S(S(S(T[{\ast}])T[{\ast}]S({\ast}{\ast}{\ast}))) 
- S(S(S(T[{\ast}])S(T[{\ast}]){\ast}{\ast}{\ast})),
\\
R_5 = R_4 \circ_5 T &=
S(S(T[{\ast}])S(T[{\ast}])S({\ast}{\ast}T[{\ast}]))
- S(S(T[{\ast}]S(T[{\ast}])S({\ast}{\ast}T[{\ast}]))) 
\\
&\quad {}
- S(S(S(T[{\ast}])T[{\ast}]S({\ast}{\ast}T[{\ast}]))) 
- S(S(S(T[{\ast}])S(T[{\ast}]){\ast}{\ast}T[{\ast}])),
\\
R_6 = R_5 \circ_3 T &=
S(S(T[{\ast}])S(T[{\ast}])S(T[{\ast}]{\ast}T[{\ast}]))
- S(S(T[{\ast}]S(T[{\ast}])S(T[{\ast}]{\ast}T[{\ast}]))) 
\\
&\quad {}
- S(S(S(T[{\ast}])T[{\ast}]S(T[{\ast}]{\ast}T[{\ast}]))) 
\\
&\quad {}
- S(S(S(T[{\ast}])S(T[{\ast}])T[{\ast}]{\ast}T[{\ast}])).
\end{align*}
\normalsize
\end{example}

\begin{remark}
Combining the monomials (in normal form) in the consequences of the RB relations with the monomials
(in normal form) in the image of the rewriting morphism gives 53521 distinct monomials.
The unique monomial which appears in the image of the rewriting morphism 
but does not appear as a monomial in the consequences of the RB relations is
$T[S({\ast})]T[S({\ast})]{\ast}T[S({\ast})]T[S({\ast})]$.
\end{remark}


\subsubsection*{Computing the $ST$ normal form}

In general the consequences of the RB relations will not be in normal form with respect to 
the commutativity of $S$ and $T$.
By normal form we mean that every occurrence of $S(T[x])$ must be replaced by $T[S(x)]$.
An algorithm for computing the normal form ($ST \rightarrow TS$)
of a monomial in $\au_{3,2}$ appears in Figure \ref{algST}.
The pointer $p$ starts at the right end of the string
and moves to the left, searching for occurrences of the substring $S(T[$,
and the corresponding right delimiters $])$.
We then replace $S(T[$ by $T[S($ and $)]$ by $])$, but note that this may introduce a new occurrence
of $S(T[$ to the right of the pointer, which requires that the pointer be reset
to the right end of the string.

\begin{example}
We present some examples of normal form computation:
\begin{alignat*}{2}
&S({\ast}{\ast}S(T[T[T[T[{\ast}]]]])S({\ast})S({\ast})) \quad &\longmapsto \quad 
&S({\ast}{\ast}T[T[T[T[S({\ast})]]]]S({\ast})S({\ast})),
\\
&S(T[S(T[{\ast}])])S(T[{\ast}{\ast}T[{\ast}]])S({\ast}) \quad &\longmapsto \quad 
&T[T[S(S({\ast}))]]T[S({\ast}{\ast}T[{\ast}])]S({\ast}), 
\\
&T[S({\ast})S(T[{\ast}])S(S(T[T[{\ast}{\ast}{\ast}]]))] \quad &\longmapsto \quad 
&T[S({\ast})T[S({\ast})]T[T[S(S({\ast}{\ast}{\ast}))]]], 
\\
&T[T[S({\ast})S({\ast})S({\ast})]S(T[{\ast}])T[{\ast}]] \quad &\longmapsto \quad 
&T[T[S({\ast})S({\ast})S({\ast})]T[S({\ast})]T[{\ast}]],
\\
&S(S({\ast})){\ast}T[S(T[S({\ast})])]T[{\ast}]T[{\ast}] \quad &\longmapsto \quad 
&S(S({\ast})){\ast}T[T[S(S({\ast}))]]T[{\ast}]T[{\ast}],
\\
&S(T[{\ast}]){\ast}S(T[S({\ast})]T[{\ast}]T[S({\ast})]) \quad &\longmapsto \quad 
&T[S({\ast})]{\ast}S(T[S({\ast})]T[{\ast}]T[S({\ast})]),
\\
&S(T[S(T[S(S({\ast}))])]T[{\ast}]T[{\ast}]){\ast}{\ast} \quad &\longmapsto \quad 
&S(T[T[S(S(S({\ast})))]]T[{\ast}]T[{\ast}]){\ast}{\ast},
\\
&T[{\ast}]T[S(T[{\ast}])]T[{\ast}S({\ast})S(S({\ast}))] \quad &\longmapsto \quad 
&T[{\ast}]T[T[S({\ast})]]T[{\ast}S({\ast})S(S({\ast}))].
\end{alignat*}
\end{example}

\begin{figure}
\begin{enumerate}
\item
$m_2 := m$
\item
$\ell := \text{length}( m_2 )$
\item
$p := \ell$ 
\item
While $p > 0$ do:
  \begin{enumerate}
  \item
  $\texttt{changed} := \text{false}$
  \item
  $v := m_2[ p \cdots \min(\ell,p+3) ]$
  \item
  If $v$ = ``$S(T[$'' then
    \begin{enumerate}
    \item[]
    (Find matching right parenthesis.)
    \item
    $\texttt{sl} := 1$, $\texttt{sr} := 0$, $i := p+1$
    \item
    While $\texttt{sl} > \texttt{sr}$ and $i < \ell$ do: 
      \begin{itemize}
      \item
      $i := i + 1$
      \item 
      If $m_2[i]$ = ``('' then $\texttt{sl} := \texttt{sl} + 1$
      \item
      If $m_2[i]$ = ``)'' then $\texttt{sr} := \texttt{sr} + 1$
      \end{itemize}
    \item[]
    (Find matching right bracket.)
    \item
    $\texttt{tl} := 1$, $\texttt{tr} := 0$, $j := p+3$
    \item
    While $\texttt{tl} > \texttt{tr}$ and $j < \ell$ do: 
      \begin{itemize}
      \item
      $j := j + 1$ 
      \item 
      If $m_2[j]$ = ``['' then $\texttt{tl} := \texttt{tl} + 1$
      \item
      If $m_2[j]$ = ``]'' then $\texttt{tr} := \texttt{tr} + 1$
      \end{itemize}
    \item[]
    (Right bracket immediately precedes right parenthesis.)
    \item
    If $j = i-1$ then 
      \begin{itemize}
      \item
      $\texttt{changed} := \text{true}$
      \item
      $m_2 := \text{concatenate}($ 
      \item[]
      $m_2[1{\cdots}p{-}1]$, ``$T[S($'', $m_2[p{+}4{\cdots}j{-}1] )$, ``$)]$'',
      $m_2[j{+}2{\cdots}\ell]$
      )
      \end{itemize}
    \end{enumerate}
  \item
  If $\texttt{changed}$ then $p := \ell$ else $p := p - 1$
  \end{enumerate}
\item
Return $m_2$
\end{enumerate}
\vspace{-8pt}
\caption{Algorithm to compute the normal form ($ST \rightarrow TS$)}
\label{algST}
\end{figure}


Before proceeding to the main result of this section (Theorem \ref{quadritheorem}), 
we present two results about the auxiliary operations.
	
\begin{proposition}
\label{totalassociative}
The total sum operation for the ternary analog of quadri-algebras is associative.
\end{proposition}

\begin{proof}
We assume that $S$ is a ternary Rota-Baxter operator for the ternary associative operation $xyz$.
Proposition \ref{sumoperationassociative} for $n = 3$ and $R = S$ shows that the sum operation
is associative; we write it as
\begin{equation}
\label{bracketoperation}
\langle x, y, z \rangle := x S(y) S(z) + S(x) y S(z) + S(x) S(y) z.
\end{equation}
We assume that $T$ is another RB operator for $xyz$ and that $ST = TS$.
We will show that $T$ is an RB operator for $\langle x, y, z \rangle$.
Since $\langle x,y,z \rangle$ is associative, another application of Proposition \ref{sumoperationassociative}
will show that $\{ x, y, z \}$ is associative where
\[
\{ x, y, z \}
:=
\langle x, T[y], T[z] \rangle + \langle T[x], y, T[z] \rangle + \langle T[x], T[y], z \rangle
\]
Expanding $\{ x, y, z \}$ using \eqref{bracketoperation} shows that it is the total sum operation.

The claim that $T$ is a ternary RB operator for $\langle  x, y, z  \rangle $ states that
\begin{align}
\label{TRB}
&
\langle  T[x], T[y], T[z]  \rangle 
=
\\
&\quad
T[ \langle  x, T[y], T[z]  \rangle  ] +
T[ \langle  T[x], y, T[z]  \rangle  ] +
T[ \langle  T[x], T[y], z  \rangle  ].
\notag
\end{align}
Using \eqref{bracketoperation} and $ST = TS$, the left side of \eqref{TRB} equals
\begin{align*}
\label{leftside}
&
L := 
\langle  T[x], T[y], T[z]  \rangle  = 
\\
&\quad
T[x] T[S(y)] T[S(z)] + T[S(x)] T[y] T[S(z)] + T[S(x)] T[S(y)] T[z].
\end{align*}
Similarly, the three terms on the right side of \eqref{TRB} equal
\begin{align*}
&
R_1 := 
T[ \langle  x, T[y], T[z]  \rangle  ] =
\\
&\quad
T[ x T[S(y)] T[S(z)] ] + T[ S(x) T[y] T[S(z)] ] + T[ S(x) T[S(y)] T[z] ],
\\
&
R_2 := 
T[ \langle  T[x], y, T[z]  \rangle  ] = 
\\
&\quad
T[ T[x] S(y) T[S(z)] ] + T[ T[S(x)] y T[S(z)] ] + T[ T[S(x)] S(y) T[z] ],
\\
&
R_3 := T[ \langle  T[x], T[y], z  \rangle  ] =
\\
&\quad
T[ T[x] T[S(y)] S(z) ] + T[ T[S(x)] T[y] S(z) ] + T[ T[S(x)] T[S(y)] z ].
\end{align*}
Since $T$ is an RB operator for $xyz$, and considering the product
$x S(y) S(z)$, the first term of $L$ equals the sum of the first terms of $R_1, R_2, R_3$:
\begin{align*}
&
T[x] T[S(y)] T[S(z)]
=
\\
&\quad
T[ x T[S(y)] T[S(z)] ]
+
T[ T[x] S(y) T[S(z)] ]
+
T[ T[x] T[S(y)] S(z) ].
\end{align*}
Similarly, the second (third) terms of $L$ match the second (third) terms of $R_1, R_2, R_3$
corresponding to the products $S(x) y S(z)$ and $S(x) S(y) z$.
\end{proof}

\begin{proposition}
\label{rowcoldendriform}
The row (column) sums for the ternary analog of quadri-algebras
satisfy the defining relations for the ternary analog of dendriform algebras.
\end{proposition}

\begin{proof}
We prove this for the row sums; the argument for the column sums is similar.
It suffices to express the row sums as $xU(y)U(z)$, $U(x)yU(z)$, $U(x)U(y)z$ 
for some Rota-Baxter operator $U$ on some associative algebra. 
Then the claim follows from Theorem \ref{ternarydendriformrelations}.
The row sums are
\begin{align*}
\omega_{ 1 \bullet }(x,y,z) &:=
x T[S(y)] T[S(z)] +
T[x] S(y) T[S(z)] +
T[x] T[S(y)] S(z), 
\\
\omega_{ 2 \bullet }(x,y,z) &:= 
S(x) T[y] T[S(z)] +
T[S(x)] y T[S(z)] +
T[S(x)] T[y] S(z), 
\\
\omega_{ 3 \bullet }(x,y,z) &:=
S(x) T[S(y)] T[z] +
T[S(x)] S(y) T[z] +
T[S(x)] T[S(y)] z.
\end{align*}
Define a new operation
(which is associative by Proposition \ref{sumoperationassociative})
\[
\langle x, y, z \rangle := xT[y]T[z] + T[x]yT[z] + T[x]T[y]z.
\]
The row sums are then 
$\langle x,S(y),S(z) \rangle$, $\langle S(x),y,S(z) \rangle$, $\langle S(x),S(y),z \rangle$.
Clearly $S$ is an RB operator for $\langle x,y,z \rangle$ as in the proof of Proposition \ref{totalassociative}
with $S$ and $T$ interchanged.
(This provides a different proof that the total sum is associative.)
\end{proof}

We are now ready to compute the quadratic relations satisfied 
(modulo the RB relations) by the rewritings of the nonassociative monomials.
The following fact allows us to relate computations over $\mathbb{Q}$
with those over $\mathbb{F}_p$ for prime $p$.

\begin{lemma}
\label{ranklemma}
Let $M$ be an integer matrix and let $p$ be a prime. 
Then
\[
\mathrm{rank}_{\,\mathbb{F}_p}(M) \le \mathrm{rank}_{\,\mathbb{Q}}(M),
\quad \mathrm{equivalently}, \quad
\mathrm{nullity}_{\mathbb{F}_p}(M) \ge \mathrm{nullity}_{\mathbb{Q}}(M).
\]
\end{lemma}

\begin{proof}
The rank is the size of the largest submatrix with a nonzero determinant.
\end{proof}

\begin{table}
\begin{alignat*}{4}
1&\colon
&\quad
( ( {\ast} {\ast} {\ast} )_{22} {\ast} {\ast} )_{11} 
&= 
( {\ast} ( {\ast} {\ast} {\ast} )_{11} {\ast} )_{22} 
&\qquad\qquad
2&\colon
&\quad
( ( {\ast} {\ast} {\ast} )_{23} {\ast} {\ast} )_{11} 
&= 
( {\ast} ( {\ast} {\ast} {\ast} )_{12} {\ast} )_{22} 
\\
3&\colon
&\quad
( ( {\ast} {\ast} {\ast} )_{32} {\ast} {\ast} )_{11} 
&= 
( {\ast} ( {\ast} {\ast} {\ast} )_{21} {\ast} )_{22} 
&\qquad\qquad
4&\colon
&\quad
( ( {\ast} {\ast} {\ast} )_{33} {\ast} {\ast} )_{11} 
&= 
( {\ast} {\ast} ( {\ast} {\ast} {\ast} )_{11} )_{33} 
\\ 
5&\colon
&\quad
( {\ast} ( {\ast} {\ast} {\ast} )_{22} {\ast} )_{22} 
&= 
( {\ast} {\ast} ( {\ast} {\ast} {\ast} )_{11} )_{33} 
&\qquad\qquad
6&\colon
&\quad
( {\ast} ( {\ast} {\ast} {\ast} )_{23} {\ast} )_{22} 
&= 
( {\ast} {\ast} ( {\ast} {\ast} {\ast} )_{12} )_{33} 
\\ 
7&\colon
&\quad
( {\ast} ( {\ast} {\ast} {\ast} )_{32} {\ast} )_{22} 
&= 
( {\ast} {\ast} ( {\ast} {\ast} {\ast} )_{21} )_{33} 
&\qquad\qquad
8&\colon
&\quad
( {\ast} ( {\ast} {\ast} {\ast} )_{33} {\ast} )_{22} 
&= 
( {\ast} {\ast} ( {\ast} {\ast} {\ast} )_{22} )_{33} 
\\ 
9&\colon
&\quad
( ( {\ast} {\ast} {\ast} )_{12} {\ast} {\ast} )_{11} 
&=
( {\ast} ( {\ast} {\ast} {\ast} )_{\bullet 1} {\ast} )_{12} 
&\qquad\qquad
10&\colon
&\quad
( ( {\ast} {\ast} {\ast} )_{\bullet 2} {\ast} {\ast} )_{21} 
&=
( {\ast} ( {\ast} {\ast} {\ast} )_{31} {\ast} )_{22} 
\\ 
11&\colon
&\quad
( ( {\ast} {\ast} {\ast} )_{13} {\ast} {\ast} )_{11} 
&=
( {\ast} ( {\ast} {\ast} {\ast} )_{\bullet 2} {\ast} )_{12} 
&\qquad\qquad
12&\colon
&\quad
( ( {\ast} {\ast} {\ast} )_{13} {\ast} {\ast} )_{11} 
&=
( {\ast} {\ast} ( {\ast} {\ast} {\ast} )_{\bullet 1} )_{13} 
\\ 
13&\colon
&\quad
( ( {\ast} {\ast} {\ast} )_{\bullet 3} {\ast} {\ast} )_{21} 
&=
( {\ast} {\ast} ( {\ast} {\ast} {\ast} )_{21} )_{33} 
&\qquad\qquad
14&\colon
&\quad
( ( {\ast} {\ast} {\ast} )_{\bullet 3} {\ast} {\ast} )_{31} 
&=
( {\ast} {\ast} ( {\ast} {\ast} {\ast} )_{31} )_{33} 
\\ 
15&\colon
&\quad
( ( {\ast} {\ast} {\ast} )_{21} {\ast} {\ast} )_{11} 
&=
( {\ast} ( {\ast} {\ast} {\ast} )_{1 \bullet} {\ast} )_{21} 
&\qquad\qquad
16&\colon
&\quad
( ( {\ast} {\ast} {\ast} )_{2 \bullet} {\ast} {\ast} )_{12} 
&=
( {\ast} ( {\ast} {\ast} {\ast} )_{13} {\ast} )_{22} 
\\ 
17&\colon
&\quad
( ( {\ast} {\ast} {\ast} )_{23} {\ast} {\ast} )_{11} 
&=
( {\ast} {\ast} ( {\ast} {\ast} {\ast} )_{\bullet 1} )_{23} 
&\qquad\qquad
18&\colon
&\quad
( ( {\ast} {\ast} {\ast} )_{31} {\ast} {\ast} )_{11} 
&=
( {\ast} ( {\ast} {\ast} {\ast} )_{2 \bullet} {\ast} )_{21} 
\\ 
19&\colon
&\quad
( ( {\ast} {\ast} {\ast} )_{31} {\ast} {\ast} )_{11} 
&=
( {\ast} {\ast} ( {\ast} {\ast} {\ast} )_{1 \bullet} )_{31} 
&\qquad\qquad
20&\colon
&\quad
( ( {\ast} {\ast} {\ast} )_{3 \bullet} {\ast} {\ast} )_{12} 
&=
( {\ast} {\ast} ( {\ast} {\ast} {\ast} )_{12} )_{33} 
\\ 
21&\colon
&\quad
( ( {\ast} {\ast} {\ast} )_{3 \bullet} {\ast} {\ast} )_{13} 
&=
( {\ast} {\ast} ( {\ast} {\ast} {\ast} )_{13} )_{33} 
&\qquad\qquad
22&\colon
&\quad
( ( {\ast} {\ast} {\ast} )_{32} {\ast} {\ast} )_{11} 
&=
( {\ast} {\ast} ( {\ast} {\ast} {\ast} )_{1 \bullet} )_{32} 
\\ 
23&\colon
&\quad
( {\ast} ( {\ast} {\ast} {\ast} )_{\bullet 2} {\ast} )_{32} 
&=
( {\ast} {\ast} ( {\ast} {\ast} {\ast} )_{31} )_{33} 
&\qquad\qquad
24&\colon
&\quad
( {\ast} ( {\ast} {\ast} {\ast} )_{13} {\ast} )_{22} 
&=
( {\ast} {\ast} ( {\ast} {\ast} {\ast} )_{\bullet 2} )_{23} 
\\ 
25&\colon
&\quad
( {\ast} ( {\ast} {\ast} {\ast} )_{\bullet 3} {\ast} )_{32} 
&=
( {\ast} {\ast} ( {\ast} {\ast} {\ast} )_{32} )_{33} 
&\qquad\qquad
26&\colon
&\quad
( {\ast} ( {\ast} {\ast} {\ast} )_{2 \bullet} {\ast} )_{23} 
&=
( {\ast} {\ast} ( {\ast} {\ast} {\ast} )_{13} )_{33} 
\\ 
27&\colon
&\quad
( {\ast} ( {\ast} {\ast} {\ast} )_{31} {\ast} )_{22} 
&=
( {\ast} {\ast} ( {\ast} {\ast} {\ast} )_{2 \bullet} )_{32} 
&\qquad\qquad
28&\colon
&\quad
( {\ast} ( {\ast} {\ast} {\ast} )_{3 \bullet} {\ast} )_{23} 
&=
( {\ast} {\ast} ( {\ast} {\ast} {\ast} )_{23} )_{33} 
\\ 
29&\colon
&\quad
( ( {\ast} {\ast} {\ast} )_{1 \bullet} {\ast} {\ast} )_{12} 
&=
( {\ast} {\ast} ( {\ast} {\ast} {\ast} )_{\bullet 2} )_{13} 
&\qquad\qquad
30&\colon
&\quad
( ( {\ast} {\ast} {\ast} )_{1 \bullet} {\ast} {\ast} )_{13} 
&=
( {\ast} {\ast} ( {\ast} {\ast} {\ast} )_{\bullet 3} )_{13} 
\\ 
31&\colon
&\quad
( ( {\ast} {\ast} {\ast} )_{\bullet 1} {\ast} {\ast} )_{21} 
&=
( {\ast} {\ast} ( {\ast} {\ast} {\ast} )_{2 \bullet} )_{31} 
&\qquad\qquad
32&\colon
&\quad
( ( {\ast} {\ast} {\ast} )_{\bullet 1} {\ast} {\ast} )_{31} 
&=
( {\ast} {\ast} ( {\ast} {\ast} {\ast} )_{3 \bullet} )_{31} 
\\ 
33&\colon
&\quad
( ( {\ast} {\ast} {\ast} )_{\bullet 2} {\ast} {\ast} )_{31} 
&=
( {\ast} {\ast} ( {\ast} {\ast} {\ast} )_{3 \bullet} )_{32} 
&\qquad\qquad
34&\colon
&\quad
( ( {\ast} {\ast} {\ast} )_{2\bullet} {\ast} {\ast} )_{13} 
&=
( {\ast} {\ast} ( {\ast} {\ast} {\ast} )_{\bullet 3} )_{23} 
\\ 
35&\colon
&\quad
( {\ast} ( {\ast} {\ast} {\ast} )_{1\bullet} {\ast} )_{23} 
&=
( {\ast} {\ast} ( {\ast} {\ast} {\ast} )_{\bullet 3} )_{23} 
&\qquad\qquad
36&\colon
&\quad
( {\ast} ( {\ast} {\ast} {\ast} )_{\bullet 1} {\ast} )_{32} 
&=
( {\ast} {\ast} ( {\ast} {\ast} {\ast} )_{3\bullet} )_{32} 
\\ 
37&\colon
&\quad
( {\ast} ( {\ast} {\ast} {\ast} )_{\bullet 3} {\ast} )_{12} 
&=
( {\ast} {\ast} ( {\ast} {\ast} {\ast} )_{\bullet 2} )_{13} 
&\qquad\qquad
38&\colon
&\quad
( {\ast} ( {\ast} {\ast} {\ast} )_{3\bullet} {\ast} )_{21} 
&=
( {\ast} {\ast} ( {\ast} {\ast} {\ast} )_{2\bullet} )_{31} 
\\ 
39&\colon
&\quad
( ( {\ast} {\ast} {\ast} )_{11} {\ast} {\ast} )_{11} 
&=
( {\ast} ( {\ast} {\ast} {\ast} )_{\bullet\bullet} {\ast} )_{11} 
&\qquad\qquad
40&\colon
&\quad
( ( {\ast} {\ast} {\ast} )_{11} {\ast} {\ast} )_{11} 
&=
( {\ast} {\ast} ( {\ast} {\ast} {\ast} )_{\bullet\bullet} )_{11} 
\\ 
41&\colon
&\quad
( ( {\ast} {\ast} {\ast} )_{\bullet\bullet} {\ast} {\ast} )_{22} 
&=
( {\ast} {\ast} ( {\ast} {\ast} {\ast} )_{22} )_{33} 
&\qquad\qquad
42&\colon
&\quad
( ( {\ast} {\ast} {\ast} )_{\bullet\bullet} {\ast} {\ast} )_{23} 
&=
( {\ast} {\ast} ( {\ast} {\ast} {\ast} )_{23} )_{33} 
\\ 
43&\colon
&\quad
( ( {\ast} {\ast} {\ast} )_{\bullet\bullet} {\ast} {\ast} )_{32} 
&=
( {\ast} {\ast} ( {\ast} {\ast} {\ast} )_{32} )_{33} 
&\qquad\qquad
44&\colon
&\quad
( ( {\ast} {\ast} {\ast} )_{\bullet\bullet} {\ast} {\ast} )_{33} 
&=
( {\ast} {\ast} ( {\ast} {\ast} {\ast} )_{33} )_{33} 
\\ 
45&\colon
&\quad
( ( {\ast} {\ast} {\ast} )_{12} {\ast} {\ast} )_{11}
&=
( {\ast} {\ast} ( {\ast} {\ast} {\ast} )_{\bullet\bullet} )_{12} 
&\qquad\qquad
46&\colon
&\quad
( ( {\ast} {\ast} {\ast} )_{21} {\ast} {\ast} )_{11} 
&=
( {\ast} {\ast} ( {\ast} {\ast} {\ast} )_{\bullet\bullet} )_{21} 
\\ 
47&\colon
&\quad
( ( {\ast} {\ast} {\ast} )_{22} {\ast} {\ast} )_{11} 
&=
( {\ast} {\ast} ( {\ast} {\ast} {\ast} )_{\bullet\bullet} )_{22} 
&\qquad\qquad
48&\colon
&\quad
( {\ast} ( {\ast} {\ast} {\ast} )_{\bullet\bullet} {\ast} )_{33} 
&=
( {\ast} {\ast} ( {\ast} {\ast} {\ast} )_{33} )_{33} 
\\ 
49&\colon
&\quad
( {\ast} ( {\ast} {\ast} {\ast} )_{\bullet\bullet} {\ast} )_{13} 
&=
( {\ast} {\ast} ( {\ast} {\ast} {\ast} )_{\bullet 3} )_{13} 
&\qquad\qquad
50&\colon
&\quad
( {\ast} ( {\ast} {\ast} {\ast} )_{\bullet\bullet} {\ast} )_{31} 
&=
( {\ast} {\ast} ( {\ast} {\ast} {\ast} )_{3\bullet} )_{31} 
\end{alignat*}
\caption{The 50 quadratic relations defining the operad $\taq$}
\label{defreltaq}
\end{table}

\begin{theorem}
\label{quadritheorem}
Over a field of characteristic 0, 
every quadratic relation satisfied by the 9 ternary operations $({\ast}{\ast}{\ast})_{ij}$
for $1 \le i, j \le 3$ is a linear combination of the 50 relations in Table \ref{defreltaq}.
\end{theorem}

\begin{proof}
In the following computations we use modular arithmetic (with $p = 101$) owing to the large size of the matrices.
We then provide explicit proofs in characteristic 0 of the relations we obtain using modular methods.

We generate the consequences of the two Rota-Baxter relations,
expressing each as a sequence of partial compositions and as a linear combination of 
two-operator monomials;
see Example \ref{excon}.
We convert each two-operator monomial to its normal form with respect to $ST$ commutativity.
We create the matrix $C$ of consequences of the RB relations,
which has 17640 rows and 53521 columns.
Entry $(i,j)$ of this matrix contains the coefficient of two-operator monomial $j$ in RB consequence $i$.
We create the matrix $X$ of rewritings which has 243 rows and 53521 columns.
Entry $(i,j)$ of this matrix is 1 if two-operator monomial $j$
is the rewriting of nonassociative monomial $i$; otherwise the entry is 0.
We construct the block matrix
\[
B = \left[ \begin{array}{cc} C & O \\ X & I \end{array} \right],
\]
which has 17883 ( = 17640 + 243 ) rows and 53764 ( = 53521 + 243 ) columns.
The matrices $O$ and $I$ are the zero and identity matrices of the obvious sizes.
Using symmetric representatives modulo $p$, the entries of $B$ belong to $\{ 0, \pm 1 \}$,
which are clearly the correct integer values.
Using modular arithmetic, we compute the row canonical form $\mathrm{RCF}(B)$; 
the rank is 17295.
This row reduction corresponds to computing the surjection from $\au_{3,2}(5;4,4)$
onto $\rb_{3,2}(5;4,4)$.
The rows of $\mathrm{RCF}(B)$ whose leading 1s appear in the right part of the matrix
($j > 53521$) represent the coefficient vectors of
the quadratic relations satisfied by the generating operations of $\na_{3,9}$.
We find that there are 50 such rows, starting with row 17246.
The entries of these rows belong to the set $\{ 0, \pm 1 \}$.
We extract the $50 \times 243$ matrix $N$ containing (the right part of) these rows.
The numbers of nonzero entries in the rows of $N$ are
$2^5,
4^{16}, 
6^{14}, 
10^{10}, 
12^4, 
18$.
The number 50 has appeared as the rank of a submatrix of $\mathrm{RCF}(B)$,
but conceptually it is the dimension of the kernel $K$ of the composition 
of the rewriting morphism with the RB surjection.
Since 50 is the nullity of the matrix representing this composition,
Lemma \ref{ranklemma} gives
\begin{equation}
\label{bound1}
\dim_\mathbb{Q} K \le 50.
\end{equation}

We now reinterpret $N$ as an integer matrix with entries $\{ 0, \pm 1 \}$.
We do this so that we can apply the LLL algorithm
for lattice basis reduction to the rows of $N$ in order to find
shorter rows which generate the same lattice;
these shorter rows correspond to relations with fewer terms.
We obtain a new lattice basis which has these numbers of nonzero entries:
$2^8, 
4^{20}, 
6^{10}, 
10^{10}, 
12^2$.
The LLL algorithm has eliminated the row of length 18, 
increased the number of rows of lengths 2 and 4, 
and decreased the number of rows of length 6.
These new shorter rows are the coefficient vectors of the relations in Table \ref{defreltaq}.
This computation with integer arithmetic will be justified 
when we provide explicit proofs in characteristic 0. 

We now extract a subset of the rows of $C$ which forms a basis of the row space.
(Compute the RCF of $C^t$ and identify the columns which contain leading 1s.)
We write $D$ for the resulting full-rank submatrix of $C$.
Let $E$ be the $53521 \times 50$ matrix in which column $j$ is the coefficient vector
of the rewriting of relation $j$ as a linear combination of two-operator monomials.
Form the block matrix $[  \,D^t \mid E \, ]$ and compute its RCF.
The column vectors in the right block express the rewritings of the 50 relations
as linear combinations of the consequences of the RB relations.
These equations are checkable in characteristic 0 as we will see below.
We verify that these 50 relations are linearly independent in characteristic 0, so
\begin{equation}
\label{bound2}
\dim_\mathbb{Q} K \ge 50.
\end{equation}
Combining this with inequality \eqref{bound1} we conclude that $\dim_\mathbb{Q} K = 50$.

The first 8 relations in Table \ref{defreltaq} express the non-injectivity 
of the rewriting morphism and have trivial proofs:
\begin{alignat*}{3}
&( ( {\ast} {\ast} {\ast} )_{22} {\ast} {\ast} )_{11} 
&\quad\longmapsto\quad
&T[ S({\ast}) ] {\ast} T[ S({\ast}) ] T[ S({\ast}) ] T[ S({\ast}) ]
&\quad\longleftmapsto\quad
&( {\ast} ( {\ast} {\ast} {\ast} )_{11} {\ast} )_{22},
\\
&( ( {\ast} {\ast} {\ast} )_{23} {\ast} {\ast} )_{11} 
&\quad\longmapsto\quad
&T[S({\ast})]T[{\ast}]S({\ast})T[S({\ast})]T[S({\ast})]
&\quad\longleftmapsto\quad
&( {\ast} ( {\ast} {\ast} {\ast} )_{12} {\ast} )_{22},
\\
&( ( {\ast} {\ast} {\ast} )_{32} {\ast} {\ast} )_{11} 
&\quad\longmapsto\quad
&T[S({\ast})]S({\ast})T[{\ast}]T[S({\ast})]T[S({\ast})]
&\quad\longleftmapsto\quad
&( {\ast} ( {\ast} {\ast} {\ast} )_{21} {\ast} )_{22},
\\
&( ( {\ast} {\ast} {\ast} )_{33} {\ast} {\ast} )_{11} 
&\quad\longmapsto\quad
&T[S({\ast})]T[S({\ast})]{\ast}T[S({\ast})]T[S({\ast})]
&\quad\longleftmapsto\quad
&( {\ast} {\ast} ( {\ast} {\ast} {\ast} )_{11} )_{33},
\\
&( {\ast} ( {\ast} {\ast} {\ast})_{22} {\ast} )_{22} 
&\quad\longmapsto\quad
&T[S({\ast})]T[S({\ast})]{\ast}T[S({\ast})]T[S({\ast})]
&\quad\longleftmapsto\quad
&( {\ast} {\ast} ( {\ast} {\ast} {\ast} )_{11} )_{33},
\\
&( {\ast} ( {\ast} {\ast} {\ast})_{23} {\ast} )_{22} 
&\quad\longmapsto\quad
&T[S({\ast})]T[S({\ast})]T[{\ast}]S({\ast})T[S({\ast})]
&\quad\longleftmapsto\quad
&( {\ast} {\ast} ( {\ast} {\ast} {\ast} )_{12} )_{33},
\\
&( {\ast} ( {\ast} {\ast} {\ast})_{32} {\ast} )_{22} 
&\quad\longmapsto\quad
&T[S({\ast})]T[S({\ast})]S({\ast})T[{\ast}]T[S({\ast})]
&\quad\longleftmapsto\quad
&( {\ast} {\ast} ( {\ast} {\ast} {\ast} )_{21} )_{33},
\\
&( {\ast} ( {\ast} {\ast} {\ast})_{33} {\ast} )_{22} 
&\quad\longmapsto\quad
&T[S({\ast})]T[S({\ast})]T[S({\ast})]{\ast}T[S({\ast})]
&\quad\longleftmapsto\quad
&( {\ast} {\ast} ( {\ast} {\ast} {\ast} )_{22} )_{33}.
\end{alignat*}
Each of the remaining 42 relations $v = w$ is proved in characteristic 0
by expressing the rewriting of $v - w$ as a linear combination
of consequences of the Rota-Baxter relations.
We then verify the equality by converting to the basis of two-operator monomials.
We list the complete results in the Appendix.
As an example we provide details of the proof of relation 50:
\[
( {\ast} ( {\ast} {\ast} {\ast} )_{\bullet\bullet} {\ast} )_{31} 
-
( {\ast} {\ast} ( {\ast} {\ast} {\ast} )_{3\bullet} )_{31}.
\]
Expand the auxiliary operations:
\begin{align*}
&
  ( {\ast} ( {\ast}{\ast}{\ast} )_{11} {\ast} )_{31}     
+ ( {\ast} ( {\ast}{\ast}{\ast} )_{12} {\ast} )_{31}     
+ ( {\ast} ( {\ast}{\ast}{\ast} )_{13} {\ast} )_{31}     
+ ( {\ast} ( {\ast}{\ast}{\ast} )_{21} {\ast} )_{31} 
\\
&{}
+ ( {\ast} ( {\ast}{\ast}{\ast} )_{22} {\ast} )_{31}     
+ ( {\ast} ( {\ast}{\ast}{\ast} )_{23} {\ast} )_{31}     
+ ( {\ast} ( {\ast}{\ast}{\ast} )_{31} {\ast} )_{31}     
+ ( {\ast} ( {\ast}{\ast}{\ast} )_{32} {\ast} )_{31} 
\\
& {}
+ ( {\ast} ( {\ast}{\ast}{\ast} )_{33} {\ast} )_{31}     
- ( {\ast} {\ast} ( {\ast}{\ast}{\ast} )_{31} )_{31}    
- ( {\ast} {\ast} ( {\ast}{\ast}{\ast} )_{32} )_{31}    
- ( {\ast} {\ast} ( {\ast}{\ast}{\ast} )_{33} )_{31}.
\end{align*}
Apply the rewriting morphism to each term, and call the result $Z$:
\begin{align*}
Z &=
  T[{\ast}]T[S({\ast}T[S({\ast})]T[S({\ast})])]S({\ast}) 
+ T[{\ast}]T[S(T[{\ast}]S({\ast})T[S({\ast})])]S({\ast}) 
\\
&\quad {}
+ T[{\ast}]T[S(T[{\ast}]T[S({\ast})]S({\ast}))]S({\ast}) 
+ T[{\ast}]T[S(S({\ast})T[{\ast}]T[S({\ast})])]S({\ast}) 
\\
&\quad {}
+ T[{\ast}]T[S(T[S({\ast})]{\ast}T[S({\ast})])]S({\ast}) 
+ T[{\ast}]T[S(T[S({\ast})]T[{\ast}]S({\ast}))]S({\ast}) 
\\
&\quad {}
+ T[{\ast}]T[S(S({\ast})T[S({\ast})]T[{\ast}])]S({\ast}) 
+ T[{\ast}]T[S(T[S({\ast})]S({\ast})T[{\ast}])]S({\ast}) 
\\
&\quad {}
+ T[{\ast}]T[S(T[S({\ast})]T[S({\ast})]{\ast})]S({\ast}) 
- T[{\ast}]T[S({\ast})]S(T[{\ast}]T[S({\ast})]S({\ast})) 
\\
&\quad {}
- T[{\ast}]T[S({\ast})]S(T[S({\ast})]T[{\ast}]S({\ast})) 
- T[{\ast}]T[S({\ast})]S(T[S({\ast})]T[S({\ast})]{\ast}).
\end{align*}
Express the five partial compositions for relation 50 in the Appendix
as linear combinations of two-operator monomials:
\begin{align}
&
- ( ( ( ( ( A \circ_2 R_T )  \circ_1 S )  \circ_2 S )  \circ_3 S )  \circ_4 S )  \circ_5 T 
=
\label{pc1}
\\
&\qquad
S({\ast})T[S({\ast})T[S({\ast})]T[S({\ast})]]T[{\ast}]
- S({\ast})T[S({\ast})]T[S({\ast})]T[S({\ast})]T[{\ast}]
\notag
\\
&\qquad
+ S({\ast})T[T[S({\ast})]S({\ast})T[S({\ast})]]T[{\ast}]
+ S({\ast})T[T[S({\ast})]T[S({\ast})]S({\ast})]T[{\ast}],
\notag
\\
&
- ( ( ( ( A \circ_2 ( T \circ R_S ) )  \circ_1 S )  \circ_3 T )  \circ_4 T )  \circ_5 T 
=
\label{pc2}
\\
&\qquad
- S({\ast})T[S({\ast})T[S({\ast})]T[S({\ast})]]T[{\ast}]
+ S({\ast})T[S({\ast}T[S({\ast})]T[S({\ast})])]T[{\ast}]
\notag
\\
&\qquad
+ S({\ast})T[S(S({\ast})T[{\ast}]T[S({\ast})])]T[{\ast}]
+ S({\ast})T[S(S({\ast})T[S({\ast})]T[{\ast}])]T[{\ast}],
\notag
\\
&
- ( ( ( ( A \circ_2 ( T \circ R_S ) )  \circ_1 S )  \circ_2 T )  \circ_4 T )  \circ_5 T 
=
\label{pc3}
\\
&\qquad
S({\ast})T[S(T[{\ast}]S({\ast})T[S({\ast})])]T[{\ast}]
+ S({\ast})T[S(T[S({\ast})]{\ast}T[S({\ast})])]T[{\ast}]
\notag
\\
&\qquad
+ S({\ast})T[S(T[S({\ast})]S({\ast})T[{\ast}])]T[{\ast}]
- S({\ast})T[T[S({\ast})]S({\ast})T[S({\ast})]]T[{\ast}],
\notag
\\
&
- ( ( ( ( A \circ_2 ( T \circ R_S ) )  \circ_1 S )  \circ_2 T )  \circ_3 T )  \circ_5 T 
=
\label{pc4}
\\
&\qquad
+ S({\ast})T[S(T[{\ast}]T[S({\ast})]S({\ast}))]T[{\ast}]
+ S({\ast})T[S(T[S({\ast})]T[{\ast}]S({\ast}))]T[{\ast}]
\notag
\\
&\qquad
+ S({\ast})T[S(T[S({\ast})]T[S({\ast})]{\ast})]T[{\ast}]
- S({\ast})T[T[S({\ast})]T[S({\ast})]S({\ast})]T[{\ast}],
\notag
\\
&
( ( ( ( ( A \circ_3 R_T )  \circ_1 S )  \circ_2 S )  \circ_3 S )  \circ_4 S )  \circ_2 T
=
\label{pc5}
\\
&\qquad
- S({\ast})T[S({\ast})]T[S({\ast})T[S({\ast})]T[{\ast}]]
+ S({\ast})T[S({\ast})]T[S({\ast})]T[S({\ast})]T[{\ast}]
\notag
\\
&\qquad
- S({\ast})T[S({\ast})]T[T[S({\ast})]S({\ast})T[{\ast}]]
- S({\ast})T[S({\ast})]T[T[S({\ast})]T[S({\ast})]{\ast}].
\notag
\end{align}
The first terms on the right sides of \eqref{pc1} and \eqref{pc2} cancel,
as do the second terms in \eqref{pc1} and \eqref{pc5},
the third term in \eqref{pc1} and the fourth term in \eqref{pc3},
and the fourth terms in \eqref{pc1} and \eqref{pc4}.
This leaves 12 two-operator monomials which coincide with $Z$.
The permutation identifying the terms of $Z$ with remaining terms from 
the partial compositions is 1, 4, 7, 2, 5, 8, 3, 6, 9, 10, 11, 12.
\end{proof}

\begin{remark}
We say that two relations in Table \ref{defreltaq} are transposes of each other 
if they are interchanged by replacing $\omega_{ij}$ by $\omega_{ji}$
(equivalently, by interchanging $S$ and $T$).
We then have the following self-transpose relations and transpose pairs.
Self-transpose:
1, 4, 5, 8, 39, 40, 41, 44, 47, 48.
Transpose pairs:
(2,3), (6,7),
(9,15), (10,16), (11,18), (12,19), (13,20), (14,21), (17,22),
(23,26), (24,27), (25,28),
(29,31), (30,32), (33,34), (35,36), (37,38),
(42,43), (45,46),
(49,50).
\end{remark}

\begin{definition}
The nonsymmetric operad $\taq$ defining the ternary analog of quadri-algebras
is generated by 9 nonassociative ternary operations satisfying the 50 relations of 
Table \ref{defreltaq}.
\end{definition}

We remark that $\taq$ may be regarded as the quotient of $\na_{3,9}$ by the ideal generated
by the 50 relations of Table \ref{defreltaq}.
The composition $f_{3,2} = s_{3,2} \circ r_{3,2}$ (Definition \ref{compositiondefinition}) 
defines a surjective operad morphism $g_{3,2}$
from $\taq$ to the suboperad of $\rb_{3,2}$ generated by the rewritings of the generating operations.
Theorem \ref{quadritheorem} identifies the full quadratic kernel of $f_{3,2}$. 
Hence the induced morphism $g_{3,2}$ has trivial quadratic kernel. 

\begin{openproblem}
It remains open whether $g_{3,2}$ is injective in all higher weights.
That is, are there any higher (cubic or above) relations satisfied by the 9 generating operations
that are not consequences of the quadratic relations?
\end{openproblem}

\begin{openproblem}
Compute a Gr\"obner basis for this operad.
Does it have a quadratic Gr\"obner basis?
\end{openproblem}

\begin{openproblem}
Is this operad Koszul?
\end{openproblem}

\begin{openproblem}
Extend these results to ternary analogs of octo-algebras \cite{Leroux},
dendriform trialgebras \cite{LR2}, and ennea-algebras \cite{Leroux2}.
\end{openproblem}


\subsubsection*{Dimension calculations for the operad $\taq$}

First recall the dimension formula for the quadri-algebra operad $\mathcal{Q}$
conjectured by Aguiar and Loday \cite[Conjecture 4.2]{AL} and proved by
Vallette \cite[Corollary 52]{V}:
\[
\dim \mathcal{Q}(n) = \frac1n \sum_{i=n}^{2n-1} \binom{3n}{n+1+i} \binom{i-1}{i-n}.
\]
This is sequence A007297 in the OEIS: the number of connected graphs 
on $n$ labeled nodes on a circle with straight-line edges that do not cross.

For the arity-graded operad $\taq$, it is obvious that
\[
\dim \taq(1) = 1, \qquad
\dim \taq(3) = 9.
\]
We have also computed the value
\[
\dim \taq(5) = 193 \, ( = 243 - 50 ). 
\]
We now describe modular computations supporting the conjecture that
\[
\dim \taq(7) = 5393.
\]
These computations are similar to the rational computations at the end of \S\ref{ternarydendriform}.
For 9 operations, we have $12 \cdot 9^3 = 8748$ cubic nonsymmetric monomials in arity 7. 
Let $\rho$ denote one of the 50 quadratic relations in Table \ref{defreltaq},
and let $\omega_{ij}$ ($1 \le i, j \le 3$) denote the 9 generating operations of $\taq$.
There are 45 partial compositions of the form $\rho \, \circ_k \, \omega_{ij}$, 
and 27 partial compositions of the form $\omega_{ij} \, \circ_k \, \rho$.
Thus we obtain 72 cubic consequences of each $\rho$, 
and hence 3600 cubic consequences $\sigma_i$ of all 50 relations.
We construct a modular matrix ($p = 101$) denoted $M$ of size $3600 \times 8748$.
The $(i,j)$ entry is the coefficient of the $j$-th cubic monomial in $\sigma_i$.
The modular rank of this matrix is 3355 and hence by Lemma \ref{ranklemma} we have
\[
\text{rank}_\mathbb{Q} M \ge 3355,
\quad
\text{which implies}
\quad
\dim_\mathbb{Q} \taq(7) \le 5393 \, ( = 8748 - 3355 ). 
\]
At the time of writing, the initial sequence 1, 9, 193, 5393 did not match any sequence in the OEIS.

\begin{openproblem}
Determine $\dim \taq(n)$ for all odd $n$.
\end{openproblem}


\section{Appendix: Exact certificates for $\taq$ relations 9 to 50}

For $9 \le i \le 50$, express relation $i$ from Table \ref{defreltaq} as $v - w = 0$.
Then item $i$ in the following list contains the rewriting of $v - w$ as a linear combination of 
consequences of the ternary Rota-Baxter relations $R_S$ and $R_T$.
These claims hold over a field of characteristic 0 and can be verified by hand as done for relation 50 
at the end of the proof of Theorem \ref{quadritheorem}:
\begin{alignat*}{2}
9\colon&\quad &\quad & 
( ( ( ( ( A \circ_2 R_S )  \circ_5 S )  \circ_1 T )  \circ_3 T )  \circ_4 T )  \circ_5 T,
\\
10\colon&\quad &\quad &
- ( ( ( ( ( A \circ_1 R_S )  \circ_5 S )  \circ_1 T )  \circ_3 T )  \circ_4 T )  \circ_5 T,
\\
11\colon&\quad &\quad &
( ( ( ( ( A \circ_2 R_S )  \circ_5 S )  \circ_1 T )  \circ_2 T )  \circ_4 T )  \circ_5 T,
\\
12\colon&\quad &\quad &
( ( ( ( ( A \circ_3 R_S )  \circ_2 S )  \circ_1 T )  \circ_2 T )  \circ_4 T )  \circ_5 T,
\\
13\colon&\quad &\quad &
- ( ( ( ( ( A \circ_1 R_S )  \circ_5 S )  \circ_1 T )  \circ_2 T )  \circ_4 T )  \circ_5 T,
\\
14\colon&\quad &\quad &
- ( ( ( ( ( A \circ_1 R_S )  \circ_4 S )  \circ_1 T )  \circ_2 T )  \circ_4 T )  \circ_5 T,
\\
15\colon&\quad &\quad &
( ( ( ( ( A \circ_2 R_T )  \circ_1 S )  \circ_3 S )  \circ_4 S )  \circ_5 S )  \circ_5 T,
\\
16\colon&\quad &\quad &
- ( ( ( ( ( A \circ_1 R_T )  \circ_1 S )  \circ_3 S )  \circ_4 S )  \circ_5 S )  \circ_5 T,
\\
17\colon&\quad &\quad &
( ( ( ( ( A \circ_3 R_S )  \circ_1 S )  \circ_1 T )  \circ_2 T )  \circ_4 T )  \circ_5 T,
\\
18\colon&\quad &\quad &
( ( ( ( ( A \circ_2 R_T )  \circ_1 S )  \circ_2 S )  \circ_4 S )  \circ_5 S )  \circ_5 T,
\\
19\colon&\quad &\quad &
( ( ( ( ( A \circ_3 R_T )  \circ_1 S )  \circ_2 S )  \circ_4 S )  \circ_5 S )  \circ_2 T,
\\
20\colon&\quad &\quad &
- ( ( ( ( ( A \circ_1 R_T )  \circ_1 S )  \circ_2 S )  \circ_4 S )  \circ_5 S )  \circ_5 T,
\\
21\colon&\quad &\quad &
- ( ( ( ( ( A \circ_1 R_T )  \circ_1 S )  \circ_2 S )  \circ_4 S )  \circ_5 S )  \circ_4 T, 
\\
22\colon&\quad &\quad &
( ( ( ( ( A \circ_3 R_T )  \circ_1 S )  \circ_2 S )  \circ_4 S )  \circ_5 S )  \circ_1 T,
\\
23\colon&\quad &\quad &
- ( ( ( ( ( A \circ_2 R_S )  \circ_1 S )  \circ_1 T )  \circ_2 T )  \circ_4 T )  \circ_5 T, 
\\
24\colon&\quad &\quad &
( ( ( ( ( A \circ_3 R_S )  \circ_1 S )  \circ_1 T )  \circ_2 T )  \circ_3 T )  \circ_5 T, 
\\
25\colon&\quad &\quad &
- ( ( ( ( ( A \circ_2 R_S )  \circ_1 S )  \circ_1 T )  \circ_2 T )  \circ_3 T )  \circ_5 T,
\\
26\colon&\quad &\quad &
- ( ( ( ( ( A \circ_2 R_T )  \circ_1 S )  \circ_2 S )  \circ_4 S )  \circ_5 S )  \circ_1 T,
\\
27\colon&\quad &\quad &
( ( ( ( ( A \circ_3 R_T )  \circ_1 S )  \circ_2 S )  \circ_3 S )  \circ_5 S )  \circ_1 T, 
\\
28\colon&\quad &\quad &
- ( ( ( ( ( A \circ_2 R_T )  \circ_1 S )  \circ_2 S )  \circ_3 S )  \circ_5 S )  \circ_1 T, 
\\
29\colon&\quad &\quad &
- ( ( ( ( ( A \circ_1 R_T )  \circ_2 S )  \circ_3 S )  \circ_4 S )  \circ_5 S )  \circ_5 T 
\\
&&&
+ ( ( ( ( ( A \circ_3 R_S )  \circ_2 S )  \circ_1 T )  \circ_2 T )  \circ_3 T )  \circ_5 T, 
\\
30\colon&\quad &\quad &
- ( ( ( ( ( A \circ_1 R_T )  \circ_2 S )  \circ_3 S )  \circ_4 S )  \circ_5 S )  \circ_4 T 
\\
&&&
+ ( ( ( ( ( A \circ_3 R_S )  \circ_2 S )  \circ_1 T )  \circ_2 T )  \circ_3 T )  \circ_4 T, 
\\
31\colon&\quad &\quad &
- ( ( ( ( ( A \circ_1 R_S )  \circ_5 S )  \circ_2 T )  \circ_3 T )  \circ_4 T )  \circ_5 T 
\\
&&&
+ ( ( ( ( ( A \circ_3 R_T )  \circ_1 S )  \circ_2 S )  \circ_3 S )  \circ_5 S )  \circ_2 T, 
\\
32\colon&\quad &\quad &
- ( ( ( ( ( A \circ_1 R_S )  \circ_4 S )  \circ_2 T )  \circ_3 T )  \circ_4 T )  \circ_5 T 
\\
&&&
+ ( ( ( ( ( A \circ_3 R_T )  \circ_1 S )  \circ_2 S )  \circ_3 S )  \circ_4 S )  \circ_2 T, 
\\
33\colon&\quad &\quad &
- ( ( ( ( ( A \circ_1 R_S )  \circ_4 S )  \circ_1 T )  \circ_3 T )  \circ_4 T )  \circ_5 T 
\\
&&&
+ ( ( ( ( ( A \circ_3 R_T )  \circ_1 S )  \circ_2 S )  \circ_3 S )  \circ_4 S )  \circ_1 T, 
\\
34\colon&\quad &\quad &
- ( ( ( ( ( A \circ_1 R_T )  \circ_1 S )  \circ_3 S )  \circ_4 S )  \circ_5 S )  \circ_4 T 
\\
&&&
+ ( ( ( ( ( A \circ_3 R_S )  \circ_1 S )  \circ_1 T )  \circ_2 T )  \circ_3 T )  \circ_4 T, 
\\
35\colon&\quad &\quad &
- ( ( ( ( ( A \circ_2 R_T )  \circ_1 S )  \circ_3 S )  \circ_4 S )  \circ_5 S )  \circ_1 T 
\\
&&&
+ ( ( ( ( ( A \circ_3 R_S )  \circ_1 S )  \circ_1 T )  \circ_2 T )  \circ_3 T )  \circ_4 T, 
\\
36\colon&\quad &\quad &
- ( ( ( ( ( A \circ_2 R_S )  \circ_1 S )  \circ_1 T )  \circ_3 T )  \circ_4 T )  \circ_5 T 
\\
&&&
+ ( ( ( ( ( A \circ_3 R_T )  \circ_1 S )  \circ_2 S )  \circ_3 S )  \circ_4 S )  \circ_1 T, 
\\
37\colon&\quad &\quad &
- ( ( ( ( ( A \circ_2 R_S )  \circ_5 S )  \circ_1 T )  \circ_2 T )  \circ_3 T )  \circ_5 T 
\\
&&&
+ ( ( ( ( ( A \circ_3 R_S )  \circ_2 S )  \circ_1 T )  \circ_2 T )  \circ_3 T )  \circ_5 T,
\\
38\colon&\quad &\quad &
- ( ( ( ( ( A \circ_2 R_T )  \circ_1 S )  \circ_2 S )  \circ_3 S )  \circ_5 S )  \circ_5 T 
\\
&&&
+ ( ( ( ( ( A \circ_3 R_T )  \circ_1 S )  \circ_2 S )  \circ_3 S )  \circ_5 S )  \circ_2 T, 
\\
39\colon&\quad &\quad &
( ( ( ( A \circ_2 ( T \circ R_S ) )  \circ_5 S )  \circ_2 T )  \circ_3 T )  \circ_5 T 
\\
&&&
+ ( ( ( ( A \circ_2 ( T \circ R_S ) )  \circ_5 S )  \circ_2 T )  \circ_4 T )  \circ_5 T 
\\
&&&
+ ( ( ( ( A \circ_2 ( T \circ R_S ) )  \circ_5 S )  \circ_3 T )  \circ_4 T )  \circ_5 T 
\\
&&&
- ( ( ( ( ( A \circ_2 R_T )  \circ_2 S )  \circ_3 S )  \circ_4 S )  \circ_5 S )  \circ_5 T, 
\\
40\colon&\quad &\quad &
( ( ( ( A \circ_3 ( T \circ R_S ) )  \circ_2 S )  \circ_2 T )  \circ_3 T )  \circ_4 T 
\\
&&&
+ ( ( ( ( A \circ_3 ( T \circ R_S ) )  \circ_2 S )  \circ_2 T )  \circ_3 T )  \circ_5 T 
\\
&&&
+ ( ( ( ( A \circ_3 ( T \circ R_S ) )  \circ_2 S )  \circ_2 T )  \circ_4 T )  \circ_5 T 
\\
&&&
- ( ( ( ( ( A \circ_3 R_T )  \circ_2 S )  \circ_3 S )  \circ_4 S )  \circ_5 S )  \circ_2 T, 
\\
41\colon&\quad &\quad &
- ( ( ( ( ( A \circ_1 R_T )  \circ_1 S )  \circ_2 S )  \circ_3 S )  \circ_5 S )  \circ_5 T 
\\
&&&
- ( ( ( ( A \circ_1 ( T \circ R_S ) )  \circ_5 S )  \circ_2 T )  \circ_3 T )  \circ_5 T 
\\
&&&
- ( ( ( ( A \circ_1 ( T \circ R_S ) )  \circ_5 S )  \circ_1 T )  \circ_3 T )  \circ_5 T 
\\
&&&
- ( ( ( ( A \circ_1 ( T \circ R_S ) )  \circ_5 S )  \circ_1 T )  \circ_2 T )  \circ_5 T, 
\\
42\colon&\quad &\quad &
- ( ( ( ( ( A \circ_1 R_T )  \circ_1 S )  \circ_2 S )  \circ_3 S )  \circ_5 S )  \circ_4 T 
\\
&&&
- ( ( ( ( A \circ_1 ( T \circ R_S ) )  \circ_5 S )  \circ_2 T )  \circ_3 T )  \circ_4 T 
\\
&&&
- ( ( ( ( A \circ_1 ( T \circ R_S ) )  \circ_5 S )  \circ_1 T )  \circ_3 T )  \circ_4 T 
\\
&&&
- ( ( ( ( A \circ_1 ( T \circ R_S ) )  \circ_5 S )  \circ_1 T )  \circ_2 T )  \circ_4 T, 
\\
43\colon&\quad &\quad &
- ( ( ( ( ( A \circ_1 R_T )  \circ_1 S )  \circ_2 S )  \circ_3 S )  \circ_4 S )  \circ_5 T 
\\
&&&
- ( ( ( ( A \circ_1 ( T \circ R_S ) )  \circ_4 S )  \circ_2 T )  \circ_3 T )  \circ_5 T 
\\
&&&
- ( ( ( ( A \circ_1 ( T \circ R_S ) )  \circ_4 S )  \circ_1 T )  \circ_3 T )  \circ_5 T 
\\
&&&
- ( ( ( ( A \circ_1 ( T \circ R_S ) )  \circ_4 S )  \circ_1 T )  \circ_2 T )  \circ_5 T, 
\\
44\colon&\quad &\quad &
- ( ( ( ( ( A \circ_1 R_T )  \circ_1 S )  \circ_2 S )  \circ_3 S )  \circ_4 S )  \circ_4 T 
\\
&&&
- ( ( ( ( A \circ_1 ( T \circ R_S ) )  \circ_4 S )  \circ_2 T )  \circ_3 T )  \circ_4 T 
\\
&&&
- ( ( ( ( A \circ_1 ( T \circ R_S ) )  \circ_4 S )  \circ_1 T )  \circ_3 T )  \circ_4 T 
\\
&&&
- ( ( ( ( A \circ_1 ( T \circ R_S ) )  \circ_4 S )  \circ_1 T )  \circ_2 T )  \circ_4 T, 
\\
45\colon&\quad &\quad &
( ( ( ( A \circ_3 ( T \circ R_S ) )  \circ_2 S )  \circ_1 T )  \circ_3 T )  \circ_4 T 
\\
&&&
+ ( ( ( ( A \circ_3 ( T \circ R_S ) )  \circ_2 S )  \circ_1 T )  \circ_3 T )  \circ_5 T 
\\
&&&
+ ( ( ( ( A \circ_3 ( T \circ R_S ) )  \circ_2 S )  \circ_1 T )  \circ_4 T )  \circ_5 T 
\\
&&&
+ ( ( ( ( ( A \circ_3 R_T )  \circ_2 S )  \circ_3 S )  \circ_4 S )  \circ_5 S )  \circ_1 T,
\\
46\colon&\quad &\quad &
( ( ( ( A \circ_3 ( T \circ R_S ) )  \circ_1 S )  \circ_2 T )  \circ_3 T )  \circ_4 T 
\\
&&&
+ ( ( ( ( A \circ_3 ( T \circ R_S ) )  \circ_1 S )  \circ_2 T )  \circ_3 T )  \circ_5 T 
\\
&&&
+ ( ( ( ( A \circ_3 ( T \circ R_S ) )  \circ_1 S )  \circ_2 T )  \circ_4 T )  \circ_5 T 
\\
&&&
+ ( ( ( ( ( A \circ_3 R_T )  \circ_1 S )  \circ_3 S )  \circ_4 S )  \circ_5 S )  \circ_2 T,
\\
47\colon&\quad &\quad &
( ( ( ( A \circ_3 ( T \circ R_S ) )  \circ_1 S )  \circ_1 T )  \circ_3 T )  \circ_4 T 
\\
&&&
+ ( ( ( ( A \circ_3 ( T \circ R_S ) )  \circ_1 S )  \circ_1 T )  \circ_3 T )  \circ_5 T 
\\
&&&
+ ( ( ( ( A \circ_3 ( T \circ R_S ) )  \circ_1 S )  \circ_1 T )  \circ_4 T )  \circ_5 T 
\\
&&&
+ ( ( ( ( ( A \circ_3 R_T )  \circ_1 S )  \circ_3 S )  \circ_4 S )  \circ_5 S )  \circ_1 T,
\\
48\colon&\quad &\quad &
- ( ( ( ( ( A \circ_2 R_T )  \circ_1 S )  \circ_2 S )  \circ_3 S )  \circ_4 S )  \circ_1 T 
\\
&&&
- ( ( ( ( A \circ_2 ( T \circ R_S ) )  \circ_1 S )  \circ_1 T )  \circ_3 T )  \circ_4 T 
\\
&&&
- ( ( ( ( A \circ_2 ( T \circ R_S ) )  \circ_1 S )  \circ_1 T )  \circ_2 T )  \circ_4 T 
\\
&&&
- ( ( ( ( A \circ_2 ( T \circ R_S ) )  \circ_1 S )  \circ_1 T )  \circ_2 T )  \circ_3 T,
\\
49\colon&\quad &\quad &
- ( ( ( ( ( A \circ_2 R_T )  \circ_2 S )  \circ_3 S )  \circ_4 S )  \circ_5 S )  \circ_1 T 
\\
&&&
- ( ( ( ( A \circ_2 ( T \circ R_S ) )  \circ_5 S )  \circ_1 T )  \circ_3 T )  \circ_4 T 
\\
&&&
- ( ( ( ( A \circ_2 ( T \circ R_S ) )  \circ_5 S )  \circ_1 T )  \circ_2 T )  \circ_4 T 
\\
&&&
- ( ( ( ( A \circ_2 ( T \circ R_S ) )  \circ_5 S )  \circ_1 T )  \circ_2 T )  \circ_3 T 
\\
&&&
+ ( ( ( ( ( A \circ_3 R_S )  \circ_2 S )  \circ_1 T )  \circ_2 T )  \circ_3 T )  \circ_4 T,
\\
50\colon&\quad &\quad &
- ( ( ( ( ( A \circ_2 R_T )  \circ_1 S )  \circ_2 S )  \circ_3 S )  \circ_4 S )  \circ_5 T 
\\
&&&
- ( ( ( ( A \circ_2 ( T \circ R_S ) )  \circ_1 S )  \circ_3 T )  \circ_4 T )  \circ_5 T 
\\
&&&
- ( ( ( ( A \circ_2 ( T \circ R_S ) )  \circ_1 S )  \circ_2 T )  \circ_4 T )  \circ_5 T 
\\
&&&
- ( ( ( ( A \circ_2 ( T \circ R_S ) )  \circ_1 S )  \circ_2 T )  \circ_3 T )  \circ_5 T 
\\
&&&
+ ( ( ( ( ( A \circ_3 R_T )  \circ_1 S )  \circ_2 S )  \circ_3 S )  \circ_4 S )  \circ_2 T.
\end{alignat*}


\section{Acknowledgement}


\subsubsection*{AI acknowledgement}

The authors used Microsoft Copilot to assist with proofreading,
expository suggestions, and bibliographic consistency checks.  All
mathematical content and conclusions were reviewed and verified by the
authors.



\end{document}